\documentclass[a4paper,9pt]{amsart}
\usepackage{amssymb,amsthm,amsmath,stmaryrd}
\usepackage{mathtools} %
\usepackage{todonotes}
\usepackage{fullpage}
\usepackage[colorlinks=true, linkcolor=blue, anchorcolor=blue, citecolor=blue, filecolor=blue, menucolor=blue, urlcolor=blue]{hyperref}
\usepackage[shortlabels]{enumitem}

\usepackage[font=small,labelfont=bf]{caption}
\usepackage[labelformat=simple]{subcaption}

\newcommand{\Q}{\mathbb{Q}}
\newcommand{\R}{\mathbb{R}}

\newcommand{\h}{\mathcal{H}}

\newcommand{\Rp}{\mathbb{RP}}
\newcommand{\pth}[1]{\left(#1\right)}
\newcommand{\ptb}[1]{\left[#1\right]}

\newcommand{\sign}{{\rm sign}}

\newcommand{\eqdef}{:=}  %

\newcommand{\defn}[1]{\emph{\color{violet}#1}} %

\DeclareMathOperator{\conv}{conv} %
\DeclareMathOperator{\dist}{dist} %
\newcommand{\aff}[1]{\langle #1 \rangle} %

\newcommand{\ball}[2]{B_{#1}( {#2} )} %
\newcommand{\sphere}[2]{S_{#1}({#2})} %

\newtheorem{theorem}{Theorem}
\newtheorem{reptheorem}{Theorem}
\newtheorem{conjecture}[theorem]{Conjecture}

\newtheorem{lemma}{Lemma}[section]
\newtheorem{proposition}[lemma]{Proposition}

\newtheorem{remark}[lemma]{Remark}

\title{An asymptotic rigidity property from the realizability of chirotope extensions}
\author{Xavier Goaoc$^1$ \and Arnau Padrol$^2$}

\address{$1$. Université de Lorraine, CNRS, INRIA, LORIA, Nancy, F-54000, France\\
  \texttt{xavier.goaoc@loria.fr}}
\address{$2$. Dept. Matem\`atiques i Inform\`atica, Universitat de Barcelona, and Centre de Recerca Matem\`atica, Spain\\
  \texttt{arnau.padrol@ub.edu}}

\thanks{The research of A. Padrol is supported by grants PID2022-137283NB-C21 and PCI2024-155081-2 funded by MCIN/AEI/10.13039/501100011033/UE, PAGCAP ANR-21-CE48-0020 of the French National Research Agency ANR, SGR GiT-UB (2021 SGR 00697) funded by the Dept. Recerca i Universitats of Generalitat de Catalunya, and the Severo Ochoa and María de Maeztu Program CEX2020-001084-M of the Spanish State Research Agency.}
\begin{document}

\begin{abstract}
  Let $P$ be a finite full-dimensional point configuration in $\R^d$. We show that if a point configuration~$Q$ has the property that all finite chirotopes realizable by adding (generic) points to $P$ are also realizable by adding points to $Q$, then $P$ and $Q$ are equal up to a direct affine transform. We also show that for any point configuration $P$ and any $\varepsilon>0$, there is a finite, (generic) extension $\widehat P$ of $P$ with the following property: if another realization $Q$ of the chirotope of $P$ can be extended so as to realize the chirotope of $\widehat P$, then there exists a direct affine transform that maps each point of $Q$ within distance $\varepsilon$ of the corresponding point of $P$.
\end{abstract}

\maketitle

\section{Introduction}

A \defn{point configuration} in $\R^d$ is a labeled set of points $P=\{p_i\}_{i\in X}$ where $X$ is a finite set, \defn{the labels}, and $p_i\in\R^d$ for each $i\in X$. Formally, $P$ is a map from $X$ to $\R^d$, that is, an element of $(\R^d)^X$. The \defn{orientation} of a $(d+1)$-tuple $(p_1,p_2, \ldots, p_{d+1})$ of points in $\R^d$ is defined as
\[ \chi(p_1,p_2, \ldots, p_{d+1}) \eqdef \sign \det \pth{\begin{matrix} p_{1} & p_{2} & \ldots & p_{{d+1}}\\ 1 &1& \ldots &1\end{matrix}}.\]
The \defn{chirotope} of a point configuration $P \in (\R^d)^X$, with $|X| \ge d+1$, is the map
\[ \chi_P \colon \left\{\begin{array}{rcl} X^{d+1} & \to & \{-,0,+\}\\ (i_1,i_2,\ldots, i_{d+1}) & \mapsto & \chi\pth{p_{i_1},p_{i_2}, \ldots, p_{i_{d+1}}}
\end{array}\right.\]
sending each $(d+1)$-tuple of $X$ to the orientation of the points they label. A \defn{realizable chirotope} on a finite set $X$ is a function $\omega:X^{d+1} \to \{-,0,+\}$, for some $d \ge 1$, that is the chirotope of some point configuration $P \in (\R^d)^X$ (i.e. $\omega = \chi_P$); in that case, $P$ is a \defn{realization} of $\omega$. We say that a point configuration $P \in (\R^d)^X$ is \defn{generic} if no $d+1$ points are contained in a common hyperplane, that is if $\chi_P$ takes values in $\{-,+\}$.

\bigskip

Let $X \subset Y$ be label sets. A point configuration $\widehat P \in (\R^d)^Y$ \defn{extends} a point configuration $P \in (\R^d)^X$ if $\widehat P_{|X} \eqdef \{\widehat p_i\}_{i \in X}$ coincides with $P$. We say that an extension $\widehat P$ of $P$ is a \defn{generic extension} of $P$ if no point $p$ of $\widehat P \setminus P$ lies on a hyperplane spanned by points of $\widehat P \setminus \{p\}$. (Note that this requirement allows $P$ to be nongeneric, and is thus stronger than asking that $\widehat P_{|Y\setminus X}$ is generic but weaker than asking that $\widehat P$ is generic.) We say that a chirotope $\chi$ is \defn{realizable on top of} a point configuration $P$ if there exists an extension of $P$ realizing~$\chi$. Our first result is the following ``rigidity'' property:

\begin{theorem}\label{t:mainog} %
Two full-dimensional point configurations $P$ and $Q$ in $\R^d$ are  directly affinely equivalent if and only if for every finite generic extension $\widehat P$ of $P$, the chirotope of $\widehat P$ is realizable on top of $Q$.
\end{theorem}

\noindent
In coarser terms, we can say that two point configurations \defn{have the same extensions} if every finite chirotope realizable on top of one is realizable on top of the other. Theorem~\ref{t:mainog} implies that for $d \ge 2$, two $d$-dimensional point configurations have the same extensions if and only if they are directly affinely equivalent. This is in sharp contrast with the situation for $d=1$, as two point configurations in the real line have the same extensions if and only if their points are in the same order. 

\bigskip

What happens if we bound from above the number of points added in the extension? ~Theorem~\ref{t:mainog} cannot generalize, not even for a single configuration $P$. Indeed, there are infinitely many convex quadrilaterals in $\R^2$ that are pairwise not affinely equivalent, whereas for every integer~$k$ there are only finitely many distinct sets of chirotopes on at most $4+k$ points. 

Finite extensions nevertheless yield a quantitative analogue of Theorem~\ref{t:mainog}, in which a single extension of bounded size can be used to discriminate all configurations that are ``at least $\varepsilon$-far away''
from a given configuration~$P$:

\begin{theorem}\label{t:maineps}
  For every $d \ge 2$ and every full-dimensional point configuration $P \in (\R^d)^X$ there exists constants $C(P)$ and $\tau(P)>0$ such that for every $0<\varepsilon \le \tau(P)$, there exists a finite generic extension $\widehat P$ of $P$, with $\widehat P \setminus P$ of size at most $C(P)\log\frac1{\varepsilon}$ with the following property: for any configuration $Q \in (\R^d)^X$ such that $\chi_{\widehat P}$ can be realized on top of $Q$, there exists a direct affine transform $\varphi$ such that $\max_{i \in X} \|p_i-\varphi(q_i)\|_2 \le \varepsilon$.
\end{theorem}

A converse to Theorem~\ref{t:maineps} is not to be expected. In particular, a simple pigeonholing argument shows that one cannot bound from above the size of an extension sufficing to discriminate any two point configurations that are in the ``$\varepsilon$-neighborhood'' of any fixed point configuration $P$.

\subsection{Proof overview}

Towards proving Theorem~\ref{t:mainog}, we consider four variants of ``extension equivalence'' for full-dimensional point configurations $P$ and $Q$ in~$\R^d$:
\begin{enumerate}[(i)]
    \item $P$ and $Q$ have the same extensions;
    \item $P$ and $Q$ have the same generic extensions;
    \item for every finite extension $\widehat P$ of $P$, the chirotope of $\widehat P$ is realizable on top of $Q$;
    \item for every finite generic extension $\widehat P$ of $P$, the chirotope of $\widehat P$ is realizable on top of $Q$.
\end{enumerate}
It is clear that any directly affinely equivalent configurations~$P$ and $Q$ satisfy all of these properties. Observe that~(iii) differs from~(i) in that it gives asymmetric roles to $P$ and $Q$. (The same goes for~(ii) and~(iv).) As illustrated on Figure~\ref{fig:nonsymmetric}, it is not obvious that~(iii) implies~(i). Yet, Theorem~\ref{t:mainog} implies that Properties~(i)-(iv) are pairwise equivalent, and capture the direct affine equivalence between $P$ and $Q$.

\begin{figure}[htpb]
	\centering
	\begin{subfigure}[t]{.3\linewidth}
		\centering
		\includegraphics[width=\linewidth]{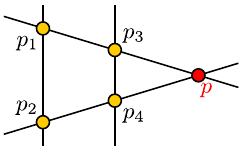}

	\end{subfigure}\qquad\qquad
	\begin{subfigure}[t]{.3\linewidth}
		\centering
		\includegraphics[width=\linewidth]{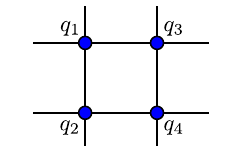}
	\end{subfigure}
	
	\caption{Two (projectively equivalent) point configurations $P=(p_1,p_2,p_3,p_4)$ and $Q=(q_1,q_2,q_3,q_4)$. The chirotope of any $1$-point extension of $Q$ is realizable on top of $P$, but the chirotope of $P \sqcup \{p\}$ is not realizable on top of $Q$.}\label{fig:nonsymmetric}
\end{figure}

\medskip

Our proofs start with an application of the classical Von Staudt constructions (see Section~\ref{s:vS}), which translate algebraic operations into point-line incidences in the plane. Lemma~\ref{l:asym2witness} then uses these constructions in a parameterized way to establish a planar version of Theorem~\ref{t:maineps} that allows nongeneric extensions.

\medskip

To handle higher dimension, we restrict our attention to small point configurations ``of corank~$1$'' (Section~\ref{s:corank1}). These configurations allow to reduce the dimensions by replacing a pair of points by the intersection of the line they span with the hyperplane spanned by the rest of the configuration. With this, Proposition~\ref{p:main} proves that in arbitrary dimension, (i) implies direct affine equivalence. This statement is not needed for our main results but we consider it of independent interest as it is much simpler to prove than Theorem~\ref{t:mainog}. 

\medskip

To move on from~(i) to~(ii) and eventually~(iv), we use a classical perturbation technique called ``scattering''. Unfortunately we are not aware of any published version of scattering that suits our purpose (among others because we work in arbitrary dimension), so we spell out two scattering arguments (Section~\ref{s:scattering}). We then generalize Lemma~\ref{l:asym2witness} to point configurations in $\R^d$ that are very generic, using ideas similar to the proof of Proposition~\ref{p:main}, while applying scatterings along the way (Section~\ref{s:verygeneric}). We then make a careful perturbation using approximations of lines and halfspaces in $\R^d$ (Section~\ref{s:perturb}) and wrap up the proofs of Theorems~\ref{t:mainog} and~\ref{t:maineps} (Section~\ref{s:proofs}). Throughout, we keep a careful account of the size of the extensions.

\subsection{Context, motivation, and related work.}

Chirotopes realizable in $\R^d$ are in a 2-to-1 correspondence with the realizable acyclic oriented matroids of rank $d+1$, and are also called \defn{labeled order types}. They have been extensively studied (see \cite{BLSWZ} for a seminal reference, and~\cite{ZAK2024} for a recently updated dynamic survey on the topic), as they provide a versatile abstract model that can be applied to many mathematical structures. They have in particular numerous applications in discrete and computational geometry~\cite{HandbookRichterGebertZiegler}, as many relevant combinatorial properties of a point configuration are captured by the associated chirotope, which is a discrete structure. 

\subsubsection*{Refinements of chirotopes.}

Chirotopes nevertheless fail to capture some important geometric properties. For example, the two configurations from Figure~\ref{fig:hexagons} have the same chirotope, but the lines they span have different intersection patterns. Several refinements of oriented matroids have been proposed to be able to distinguish finer geometric invariants of point configurations, often via oriented matroids of associated geometric objects. Some examples are \emph{allowable sequences} and \emph{sweep oriented matroids}~\cite{GP93,PadrolPhilippe}, \emph{adjoints of oriented matroids}~\cite{BachemKern1986} and {\em strong geometries}~\cite{gros2025strong}, and  CCC-systems~\cite{Knuth1992}.

\begin{figure}[htpb]
	\centering
		\includegraphics[width=.4\linewidth]{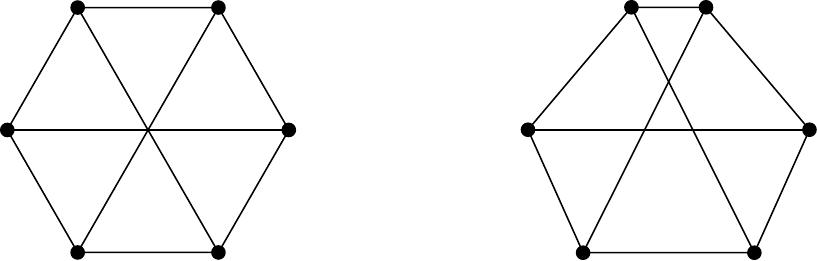}
		
	\caption{Two point configurations with the same chirotope.}\label{fig:hexagons}
\end{figure}

One way of distinguishing the two configurations of Figure~\ref{fig:hexagons} is via the one-point extensions they admit. This naturally leads to consider the finer stratifications of the set of all point configurations according not only to their order type but also to their extensions with up to $k$ points, and eventually at the limit with arbitrary many points. Our Theorem \ref{t:maineps} relates such a stratification to the Euclidean distance between point configurations modulo direct affine transforms, and Theorem~\ref{t:mainog} shows that at the limit the stratification is the one given by direct affine equivalence.

\subsubsection*{Extensions of oriented matroids.}

Extensions of oriented matroids are a very important topic by themselves, and in particular the space of all extensions of an oriented matroid has been studied~\cite{Liu2017,SturmfelsZiegler1993}. However, note that our approach concerns only those extensions that can be realized on top of a prescribed realization, which is a very different point of view that requires much distinct tools. In particular, our stratification should not be confused with the \emph{extension equivalence} proposed in \cite{BachemKernExtensions} which is done at the level of abstract oriented matroids and does not take into account the realizations.

\subsubsection*{Realization spaces.}

Theorem~\ref{t:mainog} stands in contrast to Mnëv's universality theorem~\cite{MnevT,MnevRoklin}. To see this, define, for every point configuration $P \in (\R^d)^X$ and every integer $k \ge 0$, the set $[P]_k$ as the set of point configurations $Q \in (\R^d)^{X}$ such that every chirotope of size $|X|+k$ realizable on top of $P$ is realizable on top of $Q$. On the one hand, the set $[P]_0$ is the realization space of $\chi_P$ and Mnëv's theorem asserts that in can be arbitrarily complicated (see \cite{RG95,RG99} for a survey and an alternative presentation of the proof). On the other hand, Theorem~\ref{t:mainog} states that  $[P]_\infty \eqdef \cap_{k \ge 0} [P]_k$ is the orbit of $P$ under the action of the direct invertible affine transforms of $\R^d$. Theorem~\ref{t:maineps} may suggest that restricting one's attention to the realizations that allow well-chosen finite extensions might restrict how wild the realizations can be; we conjecture that when only extensions of bounded size are considered, universality still prevails:

\begin{conjecture}\label{c:universality}
  For every $0 \le k < \infty$ and every primary basic semi-algebraic set $S$, there exists a finite point configuration $P \subset \R^2$ such that $[P]_k$ is stably equivalent to~$S$.
\end{conjecture}

\noindent
The strong geometry of Gros and Ramirez Alfonsin~\cite{gros2025strong} corresponds to the structure~$[P]_1$, and they announced us that they could prove the case k=1 of conjecture~\ref{c:universality} in the plane~\cite{RG}.

\subsubsection*{Von Staudt constructions.}

The use of the Von Staudt constructions to constrain realization spaces of oriented matroids goes back to Mn\"ev~\cite{MnevT,MnevRoklin}, and is a central piece of many of the results in the area. Our application is in particular inspired by Goodman, Pollack and Sturmfels~\cite{goodman1990intrinsic}, who used these constructions to build point configurations that that have an arbitrary large intrinsic spread, without needing to exploit the full algebraic expressivity of the Von Staudt constructions. Like them, we build extensions that are partially rigid, but towards the (different) goal of separating certain input points from infinity. For this, we need to analyze how our construction changes as one of the base points moves (see Lemma~\ref{l:asym2witness}).

\subsubsection*{Scattering techniques.}

Scattering techniques for oriented matroids were introduced by Las Vergnas in \cite{LasVergnas1986} and further refined in \cite{JMSW1989}. They were also used by Mn\"ev in the version of the Universality Theorem for uniform oriented matroids \cite{MnevRoklin}, and have proven very useful for constructing uniform oriented matroids with interesting realization spaces \cite{AdiprasitoPadrol2017,goodman1990intrinsic}. So far, they have mainly been used for planar configurations. Our presentation in Section~\ref{s:scattering} is for arbitrary dimension, but we provide a simpler version tailored for our specific needs.

\subsubsection*{Counting realizable chirotopes.}

A natural refinement of the chirotope of a point configuration $P$ is the arrangement formed by all the hyperplanes spanned by its points. When extending $P$ by a point $p$, the chirotope of $P \cup \{p\}$ only depends on the cell of the arrangement containing $p$, and distinct cells lead to distinct chirotopes. This implies that every realizable chirotope has $\Omega(n^d)$ one-point realizable extensions, which in turn implies a lower bound of $n^{d^2n+O(n/\log n)}$ for the number of chirotopes of configurations of $n$ points in $\R^d$~\cite[$\mathsection 5(i)$]{goodman1986upper}. Extending one fixed realization of a chirotope $\omega$ may, however, only give access to a small fraction of the extensions of the chirotope: for instance, the chirotope of an $n$-point sequence in convex position in the plane has $\Omega(2^n/n)$ one-point extensions (the number of self-dual 2-colored necklaces, see the discussion in~\cite{pilz2020crossing}), but each realization only gives access to $O(n^4)$ of them (see~\cite[$\mathsection 5(i)$]{goodman1986upper}). Interestingly, the lower bound on the number of realizable chirotopes obtained by repeatedly counting one-point extensions of a single realization is sharp at first order (see~\cite{alon1986number} and the very final remark of~\cite{goodman1986upper}). This suggests that for a typical realizable order type, a constant fraction of its one-point extensions can be realized on top of one and the same point configuration:

\begin{conjecture}\label{c:extensions}
    There exists some constant $c$ such that for $n \ge 3$, for a planar realizable chirotope $\chi$ of size~$n$ chosen uniformly at random, the average number of realizable $1$-point extensions (over all realizations of $\chi$) is at most $c\cdot n^4$.
\end{conjecture}

\subsubsection*{Other related results.}
Goaoc et al. established in~\cite[Theorem~4]{GHJSV2018} another rigidity result for chirotopes in the setting of probability measures representing limits of order types. In a series of works, Ne\v{s}etril, Valtr, Bárány and Pór~\cite{barany2024orientation,barany2022orientation,nevsetvril1994ramsey} established a result with a similar flavour as Theorem~\ref{t:maineps}. Specifically, let $P \in (\R^2)^X$ be a regular $n \times n$ grid. For any point configuration $Q \in (\R^2)^X$ that ``almost realizes $\chi_P$'', in the sense that $\chi_Q$ agrees with $\chi_P$ on all entries that are nonzero in $\chi_P$, there exists a projective transform $\mu$ such that $\max_{i \in X} \|p_i - \mu(q_i)\|_2 = O(1/n)$. Neither of these results appear to imply or be implied by Theorems~\ref{t:mainog} or~\ref{t:maineps}, and the proof techniques are different.

\section{Terminology and preliminaries}

We write $\|\cdot\|_2$ for the Euclidean norm of $\R^d$. Given a point configuration $P$, we write $\aff{P}$ for its \defn{affine hull}. We say that a point configuration $P \in (\R^d)^X$ is \defn{full-dimensional} if $\aff{P}=\R^d$. 

\bigskip
Two point configurations $P \in(\R^d) ^X$ and $Q \in (\R^d)^X$ are \defn{affinely equivalent} if there exists an invertible affine transformation $\phi:\aff{P}\to\aff{Q}$ such that $q_i = \phi(p_i)$ for all $i \in X$. If $\aff{P}=\aff{Q}$, we say that $\phi$ is direct if it is orientation-preserving.

\bigskip

We write $X\sqcup Y$ to denote the disjoint union of $X$ and $Y$. Given $P \in (\R^d)^X$ and $P' \in (\R^d)^Y$, we write $P \sqcup P'$ for the element in $(\R^d)^{X \sqcup Y}$ that restricts to $P$ and $P'$ on $X$ and $Y$, respectively (under the canonical inclusions of $X$ and $Y$ into $X\sqcup Y$). The next lemma is elementary but turns out to be quite useful.

\begin{lemma}\label{l:extrest}
  Let $\widehat P \in (\R^d)^{X \sqcup Y}$ be an extension of $P \in (\R^d)^X$ and $\widehat Q \in (\R^d)^{X \sqcup Y}$ an extension of $Q \in (\R^d)^X$. If every chirotope realizable on top of $\widehat P$ is realizable on top of $\widehat Q$, then every chirotope realizable on top of $P$ is realizable on top of $Q$.
\end{lemma}
\begin{proof}
  Let $\chi$ be a chirotope realizable on top of $P$ by some point configuration $P \sqcup P' \in (\R^d)^{X \sqcup Z}$. The point configuration $P'' \eqdef P \sqcup P' \sqcup (\widehat P_{|Y})$ is an extension of $\widehat P$. Hence, there exists an extension $Q''$ of $\widehat Q$ such that $\chi_{P''}=\chi_{Q''}$. The point configuration $Q''_{|X \sqcup Z}$ extends $Q = \widehat Q_{|X}$ and its chirotope equals $\chi_{P''_{|X \sqcup Z}} = \chi$.
\end{proof}

\bigskip

Recall that the \defn{cross-ratio} of an ordered quadruple $(p,q,r,s)$ of aligned points in Euclidean space is denoted by $(p,q;r,s)$ and defined as the following quotient of products of signed distances:
\begin{equation}\label{eq:crossratio} (p,q;r,s) = \frac{pr \cdot qs}{qr \cdot ps}.\end{equation}
Also recall that the cross-ratio of an ordered quadruple of four aligned points in projective space is well-defined, as all its Euclidean projections have the same cross-ratio. We use cross-ratios to parameterize projective lines: given a projective line $\ell$ equipped with a projective basis $b = (b_0,b_1,b_\infty)$, for $t \in \R \cup \{\infty\}$ we let $b(t)$ denote the unique point of $\ell$ that satisfies 
\[(b_0,b_\infty ; b(t),b_1)= \frac{b_0b(t) \cdot b_\infty b_1}{b_\infty b(t) \cdot b_0 b_1} = t.\] 
(In particular, $b(0)=b_0$, $b(1) = b_1$ and~$b(\infty)=b_\infty$.)

\section{Tools: Von Staudt constructions}\label{s:vS}

In this section we present a folklore property of point configurations in the plane (Lemma~\ref{l:VonStaudt}). Our proof hinges on the classical constructions of Von Staudt~\cite{Staudt} that translate algebraic operations into point-line incidences. We give here a presentation of these constructions tailored to our needs, and refer to~\cite[Sec.~5.6 \& 8.4]{Richter-Gebert2011} for a more complete exposition.

\subsection{Algebra from line incidences}

We say that two point configurations $P, Q \in (\R^2)^X$ \defn{have the same colinearities} if for every triple of indices $i, j, k \in X$, the points $p_i$, $p_j$ and $p_k$ are colinear if and only if the points $q_i$, $q_j$ and $q_k$ are colinear. Given a point configuration $P \in (\R^2)^X$ and an extension $\widehat P \in (\R^2)^{X \sqcup Y}$ of $P$, we say that~$\widehat P$ is a \defn{constructible extension} of~$P$ if there exists an ordering of $X \sqcup Y$ such that (i) every element of $X$ precedes every element of $Y$, and (ii) no point with label in $Y$ lies on more than two lines spanned by points with lower indexes. 

\bigskip

\begin{figure}[htpb]
\centering
\begin{subfigure}[t]{.43\linewidth}
\centering
\includegraphics[width=\linewidth]{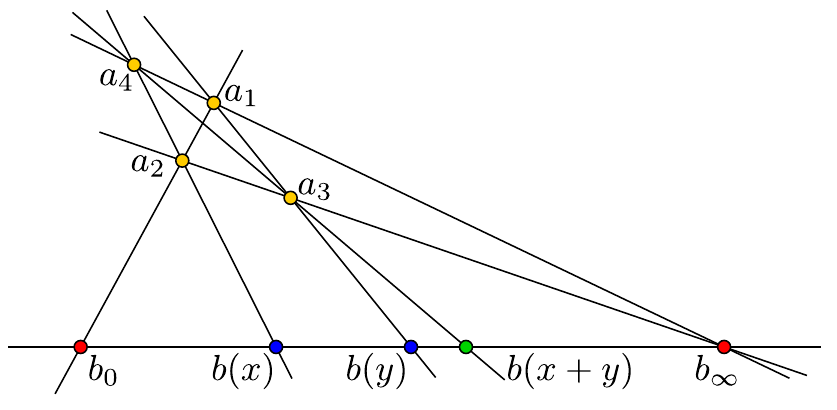}
\caption{Addition}
\end{subfigure}\qquad
\begin{subfigure}[t]{.43\linewidth}
\centering
\includegraphics[width=\linewidth]{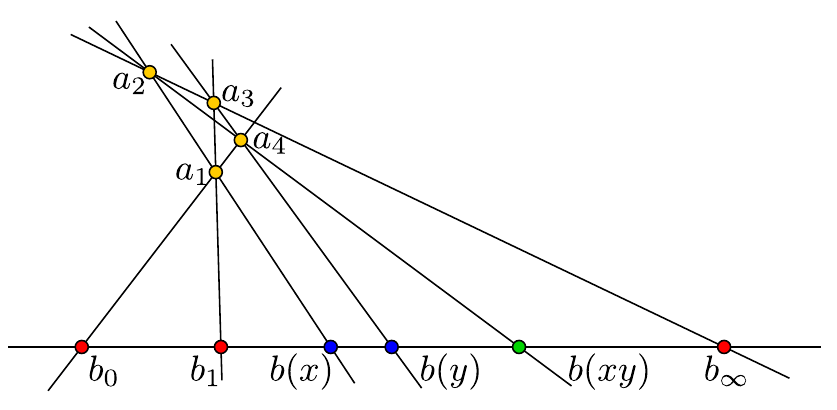}
\caption{Multiplication}
\end{subfigure}
\begin{subfigure}[t]{.43\linewidth}
\centering
\includegraphics[width=\linewidth]{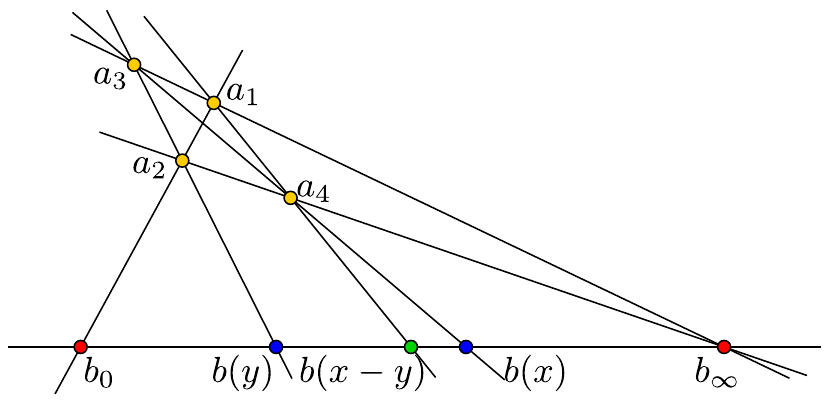}
\caption{Subtraction}
\end{subfigure}\qquad
\begin{subfigure}[t]{.43\linewidth}
\centering
\includegraphics[width=\linewidth]{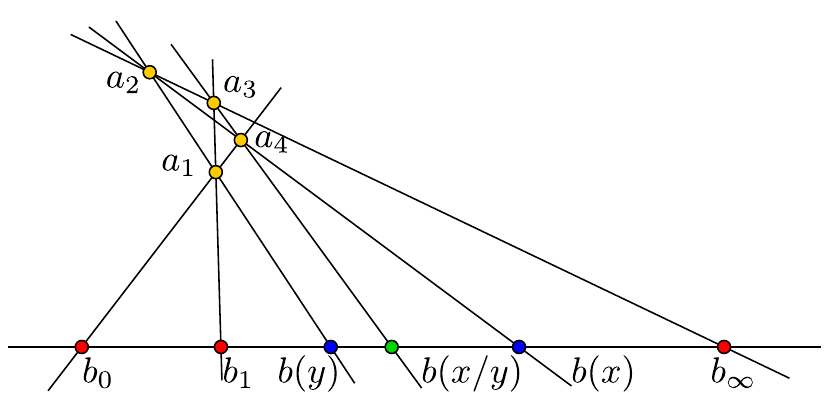}
\caption{Division}
\end{subfigure}

\caption{The construction of Von Staudt. In each configuration, the parameterization is relative to the projective basis $B=(b_0,b_1,b_\infty)$ formed by the red points (in the addition and substraction constructions, the point $b_1$ plays no role so it is omitted). In each construction, the parameters of the blue points are chosen freely and the parameter of the green point is determined by the incidences.
}\label{fig:VonStaudt}
\end{figure}

\noindent
We depict in Figure~\ref{fig:VonStaudt}(a) the classical construction of Von Staudt for the addition, which we now explain. Formally, an \defn{addition gadget} is a configuration of 9 points with the same collinearities as the construction of Figure~\ref{fig:VonStaudt}(a), and that coincides with this construction on $b_0$ and $b_\infty$. Addition gadgets have the following properties:
\begin{description}
\item[Rigidity] For any $x,y \in \R$, if in an addition gadget the blue points are $b(x)$ and $b(y)$, then the green point must be $b(x+y)$.
\item[Genericity] For any countable set $L$ of lines and any countable set $S$ of points outside of $\ell$, and for every $x, y \in \R$, there exists an addition gadget where the blue points are $b(x)$ and $b(y)$, where no yellow point lies on any line of
  $L$, and where no point of $S\setminus \{b_0, b_\infty, b(x),b(y), b(x+y)\}$ is colinear with two yellow points.
\item[Constructibility] An addition gadget is a constructible extension of $\{b_0,b_\infty,b(x),b(y)\}$.
 \end{description}
Let us stress that the construction retains these properties when the points $b(x)$ and $b(y)$ coincide; in that case, the green point is $b(2x)$. Note that for any $x,y \in \R$, we can also use an addition gadget by first setting one blue point to be $b(x)$ and the green point to be $b(y)$, so that the rigidity property implies that the other blue point is $b(x-y)$, thereby obtaining a \defn{subtraction gadget}. We can similarly define a \defn{multiplication gadget} and a \defn{division gadget} where the rigidity property becomes: For any $x,y \in \R$, if the blue points are $b(x)$ and $b(y)$, then the green point must be $b(x\cdot y)$ and $b(x/y)$, respectively. The proofs of all these assertions are standard and we refer for example to Richter-Gebert~\cite[Sec.~5.6 \& 8.4]{Richter-Gebert2011} for the details.

\subsection{A folklore (projective) lemma}

The Von Staudt constructions have been designed to relate algebra and incidence geometry. We use them in an elementary way: we fix a basis $(p_0,p_1,p_\infty)$ of a projective line $\bar \ell$, and build, for any goal~$g \in \Q$, an extension whose colinearities ``determine'' the point~$p(g)$. Our intention is to transport this construction to the affine setting, so we provision for some arbitrary point $f$ of $\ell$ to be ``forbidden'' (it will be the point at infinity in that affine chart). In order to measure the size of this extension, let us define the \defn{arithmetic complexity} of a rational $g$ as the minimum $k$ such that there exists a sequence $(u_0, u_1, u_2, \ldots, u_{k+1})$ of rationals such that
\begin{itemize}
  \item[(a)] $u_0=0$, $u_1 =1$, and $u_{k+1} = g$,
  \item[(b)] for every $2 \le i \le k+1$, there exist integers $0 \le s,t <i$ such that $u_t \neq 0$ and $u_i \in \{u_s+u_t,
  u_s-u_t, u_su_t, u_s/u_t\}$.
\end{itemize}

\begin{lemma}[Von Staudt~\cite{Staudt}]\label{l:VonStaudt}
  Let $P \in (\R^2)^X$ be a point configuration containing three distinct aligned points $p_0$, $p_1$ and $p_\infty$. Let $p(\cdot)$ denote the parameterization of the projective completion of $\aff{\{p_0,p_1\}}$ defined by the projective basis $(p_0,p_1,p_{\infty})$. For any $g \in \Q$ and any $f\in \R\setminus\{0,1,g\}$ such that $p(f) \notin P$, there exists a constructible extension $\widehat P$ of $P$ such that (i) $p(f) \notin \widehat P$, (ii) $p(g) \in \widehat P$, and (iii) any extension of $\{p_0,p_1,p_{\infty}\}$ with the same colinearities as $\widehat P$ contains the point $p(g)$ with the same label as in~$\widehat P$. Moreover, the size of $\widehat P \setminus P$ is at most $10c+15$ where $c$ is the arithmetic complexity of $g$.
\end{lemma}
\begin{proof}
  We first build a sequence $U = (u_0, u_1, u_2, \ldots, u_{k+1})$ of rationals that has properties~(a), (b) and
  \begin{itemize}
  \item[(c)] $u_i \neq f$ for every $0 \le i \le  k+1$.
  \end{itemize}
  Indeed, let $c$ denote the arithmetic complexity of $g$ and let $u_0,u_1,\ldots, u_{c+1}$ be a sequence satisfying the conditions~(a) and~(b) above. If $f$ does not appear in that sequence we are done with $k=c$. Otherwise, we can replace $f$ by $-f$, and replace every use of $f$ to define later terms by a use of $-f$. This increases the sequence length by at most $3$ (to produce $-f$ instead of $f$, which may require to produce the opposite of the two terms $u_s, u_t$ used to compute $f$) plus at most $c$ (to replace each use of $f$ by an equivalent use of $-f$, which may require to produce the opposite of the other term in the operation). Altogether we obtain a sequence as desired with $k \le 3+2c$.

\bigskip

We construct a sequence $\widehat P_1, \widehat P_2, \ldots, \widehat P_{k-1}$ of point configurations, starting with $\widehat P_0 = \widehat P_1 \eqdef P \supseteq \{p(0),p(1),p(\infty)\}$, and each $\widehat P_i$ extending $\widehat P_{i-1}$ so as to construct $p(u_i)$ via a Von Staudt gadget. Formally, we write $\widehat P_\alpha = \{p_i\}_{i\in X_\alpha}$, so $X_1=X$. The sequence satisfies the following properties for every $1 \le \alpha \le k-1$:
  \begin{itemize}
  \item[(a')] for every integer $0 \le j \le \alpha$, there is some $i_j\in X_\alpha$ such that $p_{i_j} = p(u_j)$, 
  \item[(b')] for any extension $Q = \{q_i\}_{i \in X_\alpha}$ of $P$  with the same colinearities as $P_\alpha$, we must have $q_{i_j} = p(u_j)$ for every $0 \le j \le \alpha$,
  \item[(c')] $\widehat P_\alpha$ does not contain $p(f)$,
  \item[(d')] $\widehat P_\alpha$ is a constructible extension of $\widehat P_{\alpha-1}$.
  \end{itemize}
  Note that $\widehat P_1$ is already defined and satisfies (a')-(d'). We describe the construction of $\widehat P_\alpha$ and prove that it has properties (a')-(d') for $\alpha \ge 2$ by induction on $\alpha$. Suppose now that $\widehat P_{\alpha-1}$ has been constructed and satisfies (a')-(d'). By Property~(b) of $U$, there exist integers $0 \le s,t \le \alpha-1$ such that $u_\alpha \in \{u_s+u_t, u_s-u_t, u_su_t, u_s/u_t\}$. Depending on which operation produces $u_\alpha$, we consider an addition, substraction, multiplication or division gadget, with $x = u_s$ and $y=u_t$, so as to construct $p(u_\alpha)$. Here we use the genericity property of the gadgets to make sure that the yellow points do not form any colinearity with the points of $\widehat P_{\alpha-1}$ that are not in the gadget. We then define $\widehat P_\alpha$ as the extension of $\widehat P_{\alpha-1}$ by the yellow and green points of the gadget. Properties~(a') and~(b') are inherited from $\widehat P_{\alpha-1}$ for all $0 \le j \le \alpha -1$. For $j=\alpha$, they follow from the rigidity property of the gadget for $i_\alpha$ the index of the green point added to $P_\alpha$. Property~(c') follows from Property~(c) of $U$. And property~(d') follows from the constructibility property of each gadget, and the fact that, by genericity, their points do not add colinearities with the other points that could break constructibility. 

  \bigskip

  By induction, $\widehat P_\alpha$ is defined and satisfies (a')-(d') for every $1 \le \alpha \le k$. We let $\widehat P \eqdef \widehat P_{k-1}$ and claim that it has the desired properties. Indeed, $\widehat P$ is a constructible extension of $P$ because the sequence $\{\widehat P_\alpha\}_{1 \le \alpha \le k-1}$ satisfies~(d'), and a constructible extension of a constructible extension of $P$ is again a constructible extension of $P$. Moreover, property~(i) for $\widehat P$  follows from property~(c') for $\widehat P_{k-1}$, property~(ii) for $\widehat P$  follows from $u_{k-1}=g$ and property~(a') for $\widehat P_{k-1}$, and property~(iii) for $\widehat P$ follows from property (b') for the sequence $\{\widehat P_\alpha\}_{1 \le \alpha \le k-1}$. It is clear from Figure~\ref{fig:VonStaudt} that each extension step adds 5 points (the four yellow points and the green one), so the size of $\widehat P$ is exactly the size of $P$ plus $5k$.
\end{proof}

\section{A planar extension that can only be realized on top of ``affinely close'' point sets}\label{s:basic}

With Lemma~\ref{l:VonStaudt} at our fingertips, we can essentially
give a version of Theorem~\ref{t:maineps} for $d=2$ that allows
nongeneric extensions.

\begin{lemma}\label{l:asym2witness}
  For any three distinct aligned points $p_0$, $p_1$, $p$ in $\R^2$, there exists a constant $c$ such that the following holds: For any point configuration $P \in (\R^2)^X$ containing $p_0$, $p_1$ and $p$, and for every $\varepsilon>0$, there exists a finite constructible extension $\widehat P$ of $P$, with $\widehat P \setminus P$ of size at most $20\log\frac1{\varepsilon}+c$, such that the chirotope of $\widehat P$ cannot be realized on top of $\{p_0,p_1,q\}$ for any $q$ with $\|p-q\|_2>\varepsilon$.
\end{lemma}
\begin{proof}
  Let us fix $\varepsilon>0$. We let $\ell$ denote the line $\aff{\{p_0,p_1\}}$. 
  We start with some simplifying observations:
\begin{itemize}
\item we can assume that $p\notin \{p_0,p_1\}$, as otherwise the statement is immediate since $\chi_P$ itself cannot be realized on top of $\{p_0,p_1,q\}$ whenever $\|p-q\|_2>0$,
\item we can restrict our attention to $q \in \ell$, as otherwise $\chi_P$ itself cannot be realized on top of $\{p_0,p_1,q\}$,
\item we can further assume that $q\notin \{p_0,p_1\}$, as otherwise $\chi_{P'}$ cannot be realized on top of $\{p_0,p_1,q\}$, where $P'$ is the extension of $P$ with the midpoints of the segments $p_0p$ and $p_1p$ added,
\item the roles of $p_0$ and $p_1$ are symmetric so far, and the ratio of signed distances satisfy $\frac{pp_1}{p_0p_1} + \frac{pp_0}{p_1p_0} = 1$; thus, up to exchanging the roles of $p_0$ and $p_1$  we can assume that the ratio of signed distances $\frac{pp_1}{p_0p_1}$ is not in $(-\frac12,\frac12)$.
\end{itemize}
So we now set out to construct a finite extension of $P'$ whose chirotope cannot be realized on top of $\{p_0,p_1,q\}$ for any $q\in \ell\setminus \{p_0,p_1\}$ with $\|p-q\|_2\geq \varepsilon$. 

\bigskip

We embed the affine plane $\R^2$ into the projective plane $\Rp^2$ and denote by $\bar\ell$ the projective completion of~$\ell$. We write $\Omega$ for the point of $\bar\ell\setminus\ell$, that is the point of $\ell$ at infinity. For any point $x\in \ell\setminus \{p_0,p_1\}$ we put $b(x) \eqdef (p_0,p_1,x)$. Each triple $b(x)$ is a projective basis of $\bar \ell$ and determines a parametrization of $\bar\ell$ by $\R\cup\{\infty\}$. For $x \in \ell\setminus \{p_0, p_1\}$, we denote by $b(x,t)$ the point of $\bar\ell$ with parameter $t$ in the basis $b(x)$, that is the point such that the cross-ratio $(p_0,x;b(x,t),p_1)$ equals~$t$. Note that in particular, we have $b(x,\infty)=x$, $b(x,1)=p_1$, and $b(x,0)=p_0$ for every $x \in \ell\setminus \{p_0, p_1\}$. We define $\omega(x) = (p_0,x;\Omega,p_1)$, so that $\Omega =b(x,\omega(x))$. 
Replacing $p,q,s$ by $p_0,x,p_1$ and taking the limit $r\to \infty$ in \eqref{eq:crossratio}, we get $\omega(x)=\frac{xp_1}{p_0p_1}$, to be understood as a ratio of signed distances. In particular, $\omega(p) \notin (-\frac12,\frac12)$ and for every point $z\in \ell\setminus \{p_0,p_1\}$ we have $|\omega(p)-\omega(z)|=\frac{\|p-z\|_2}{\|p_0-p_1\|_2}$. We consider the interval
\[I \eqdef \{\omega(z) \colon z \in \ell\setminus\{0,1\} \text{ s.t. } \|p-z\|_2\le\varepsilon\} = \ptb{\omega(p)-\frac{\varepsilon}{\|p_0-p_1\|_2},\omega(p)+\frac{\varepsilon}{\|p_0-p_1\|_2}}\]
that is the image under $\omega$ of the points $q$ that we do not need to control.

\bigskip

We next choose two rationals $g_-<\omega(p)< g_+$ such that $0 \notin [g_-,g_+] \subseteq I$. (We spell out the specific choice of $g_-$ and $g_+$ in the very last step of the proof.) So consider any point $q \in \ell\setminus\{p_0,p_1\}$ with $\|p-q\|_2>\varepsilon$. Since the map $x\mapsto \omega(x)$ is a homeomorphism between $\ell\setminus\{p_0,p_1\}$ and $\R\setminus \{0,1\}$, we have the situation depicted in Figure~\ref{fig:separation}. In particular, since $b(p,\omega(p))=\Omega$, in the affine line $\ell=\bar \ell\setminus \Omega$, the points $b(p,g_-)$ and $b(p,g_+)$ are on different sides of $b(p,0)=p_0$, whereas the points $b(q,g_-)$ and $b(q,g_+)$ are on the same side of $b(q,0)=p_0$. Altogether, when considering the full plane we have, for any choice of $p_2 \in \R^2 \setminus \ell$ and $q_2 \in \R^2 \setminus \ell$,
\begin{equation}\label{eq:chirobc2}
  \chi(p_0,p_2,b(p,g_-)) \neq \chi(p_0,p_2,b(p,g_+)) \quad \text{and} \quad \chi(p_0,q_2,b(q,g_-)) = \chi(p_0,q_2,b(q,g_+)).
\end{equation} 

\begin{figure}[htpb]
  \centering
  \includegraphics[width=.95\linewidth]{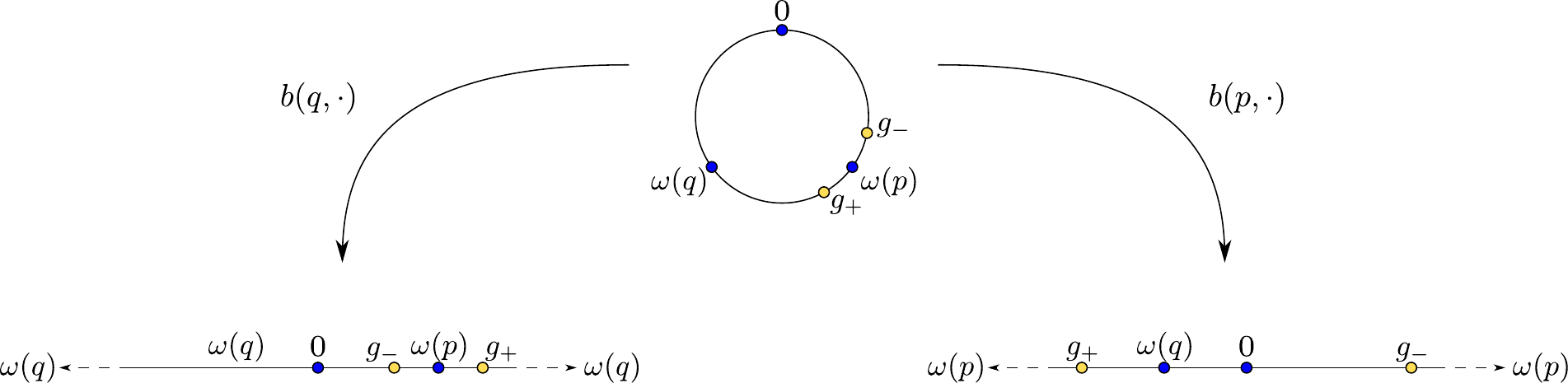}
  \caption{The maps $b(q,\cdot)$ and $b(p,\cdot)$ map $\R\cup\{\infty\}$ to the projective line $\bar \ell$. Restricted respectively to $\R\cup\{\infty\}\setminus \{\omega(q)\}$ and $\R\cup\{\infty\}\setminus \{\omega(p)\}$, they provide a map to the affine line $\ell$ (with the image of $\infty$ being respectively $p$ and $q$). This figure represents these maps schematically.
  \label{fig:separation}}
\end{figure}

\bigskip

To construct $\widehat P$ it suffices to use Lemma~\ref{l:VonStaudt} twice. First, we build  a finite constructible extension~$\widehat P^-$ of~$P$ that avoids $\Omega$ and such that any realization of~$\widehat P^-$ on top of $\{p_0,p_1,p\}$ contains $b(p,g_-)$, with the same label. (Specifically, we apply Lemma~\ref{l:VonStaudt} with parameters $(g,f, p_\infty, p_0, p_1)$ set to our current $(g_-,\omega(p), p, p_0,p_1)$.) Second, we build similarly a finite constructible extension~$\widehat P$ of~$\widehat P^-$ that avoids~$\Omega$ and such that any realization of~$\widehat P$ on top of $\{p_0,p_1,p\}$ contains~$b(p,g_+)$, with the same label. Altogether, $\widehat P$ is a finite constructible extension of~$P$, and any realization of~$\widehat P$ on top of $\{p_0,p_1,p\}$ contains $b(p,g_-)$ and $b(p,g_+)$ with the same labels. Let~$m_-$ and~$m_+$ denote these labels.

\bigskip

We claim that $\widehat P$ cannot be realized on top of $\{p_0,p_1,q_2,q\}$ for any $q_2 \in \R^2$ and any $q \in \ell\setminus\{p_0,p_1\}$ such that $\|p-q\|_2>\varepsilon$. Suppose by contradiction that such a realization $\widehat Q$ exists. Since $p_2 \notin \aff{\{p_0,p_1\}}$, we can assume that $q_2 \in \R^2 \setminus \ell$. Let $\bar\phi$ denote the unique projective transformation of $\bar\ell$ that fixes~$p_0$ and~$p_1$, and maps~$q$ to~$p$. Hence, $b(p,t) = \bar\phi(b(q,t))$ for every $t \in \R\cup\{\infty\}$. Extend $\bar\phi$ to a projective transform~$\phi$ of~$\Rp^2$ that fixes~$p_0$ and~$p_1$, and maps~$q_2$ to~$p_2$ and~$q$ to~$p$. The point configuration $\phi(\widehat Q)$ is an extension of $P$ by definition of $\phi$. Moreover, $\phi(\widehat Q)$ has the same colinearities as $\widehat P$ since $\chi_{\widehat Q} = \chi_{\widehat P}$ and $\phi$ preserves colinearities. This implies, by Property~(iii) of Lemma~\ref{l:VonStaudt}, that $\phi(\widehat Q)$ contains $b(p,g_-)$ and $b(p,g_+)$ with labels $m_-$ and $m_+$.\footnote{The careful reader will rightfully wonder what happens if $\phi$ maps some point of $\widehat Q$ to infinity. This can be taken care of by observing that the following strengthening of Property~(iii) of Lemma~\ref{l:VonStaudt} actually holds: any point configuration $\tilde P \in (\Rp^2)^Y$ that contains $\{p_0,p_1,p_{\infty}\}$ and has the same colinearities as $\widehat P$ must contain the point $p(g)$ with the same label as in~$\widehat P$. Indeed, the Von Staudt constructions are inherently projective. This is the only place in the paper where we need this strengthening, so we chose for the sake of the clarity to keep the presentation affine.} The pre-images under $\phi$, in $\widehat Q$, of these points are then $b(q,g_-)$ and $b(q,g_+)$, respectively. It follows from Equation~\eqref{eq:chirobc2} that
\[\begin{aligned}
\chi_{\widehat Q}(0,2,{m_-}) = \chi(p_0,q_2,b(q,g_-)) & = \chi(p_0,q_2,b(q,g_+)) = \chi_{\widehat Q}(0,2,{m_+})\\
\chi_{\widehat P}(0,2,{m_-}) = \chi(p_0,p_2,b(p,g_-)) & \neq \chi(p_0,p_2,b(p,g_+)) = \chi_{\widehat P}(0,2,{m_+}) \end{aligned}\]
contradicting the assumption that $\chi_{\widehat Q} = \chi_{\widehat P}$.

\bigskip

It remains to bound from above the size of $\widehat P \setminus P$. By Lemma~\ref{l:VonStaudt}, this size is governed by the arithmetic complexities of $g_-$ and $g_+$. Let us give an algorithm that produces suitable $g_-$ and $g_+$ and allows bounding their arithmetic complexities from above. We proceed in three steps:
\begin{itemize}
\item[1.] Let $a$ be the nearest integer rounding of $\omega(p) = \frac{pp_1}{p_0p_1}$ (furthest away from $0$ in case of ties). Since $\omega(p) \notin (-\frac12,\frac12)$ we have $a \neq 0$. We start our sequence by computing $a$. The number of terms is at most the  arithmetic complexity of $a$, which we denote by $c(p_0,p_1,p)$.
\item[2.] We repeatly bisect the interval $[a-1/2, a+1/2]$ in the middle, each time keeping the half that contains $\omega(p)$, until we get an interval $[a',b']$ of length at most $\frac{\varepsilon}{\|p_0-p_1\|_2}$. During the bisection, we extend the sequence by two numbers at each step: the length of the interval (initially $1$, halved at each step) and the new endpoint (obtained by adding/subtracting the updated length to/from one endpoint). This increases the length of our sequence by $1$ (to produce the integer $2$) plus $2 \cdot \lceil\log\frac{\|p_0-p_1\|_2}{\varepsilon}\rceil$.
\item[3.] If $\omega(p)$ is on the boundary of that interval, we again halve the length of the interval and add/substract it to/from an endpoint to put $\omega(p)$ back in the interior. If needed, this step increases the length of the sequence by at most~$2$.
\end{itemize}
We set $g_-$ and $g_+$ to be the endpoints of the resulting interval. That interval does not contain $0$ (since $a\neq 0$) and is contained in $I$ (since it contains $\omega(p)$ and has length at most $\frac{\varepsilon}{\|p_0-p_1\|_2}$). This sequence has length at most $c(p_0,p_1,p)+2\lceil\log\frac{\|p_0-p_1\|_2}{\varepsilon}\rceil+3$, which serves as an upper bound on the arithmetic complexity of $g_-$ and $g_+$. The same sequence can be used to construct both numbers, and therefore this bounds the number of Von Staudt gadgets. Altogether, $\widehat P \setminus P$ has size at most $20 \log \frac1{\varepsilon} + \pth{2+45 + 10c(p_0,p_1,p) +20\log \|p_0-p_1\|_2}$, where $c(p_0,p_1,p)$ is the arithmetic complexity of the nearest rounding of $\frac{pp_1}{p_0p_1}$ (away from $0$ in case of ties). (This also accounts for the midpoints of $p_0p$ and $pp_1$ added in the initial remarks.)
\end{proof}

\begin{remark}\label{r:perturb}
Note that if we move $p$ within distance $\|p_0-p_1\|_2$ of its initial position, the integer rounding of $\frac{pp_1}{p_0p_1}$ changes by at most~1, which increases the arithmetic complexity of $g_-$ and $g_+$ by at most~$1$ (in step 1, the only change). It follows that for any point $\tilde p$ such that $\|p-\tilde p\|_2 \le \|p_0-p_1\|_2$ we have $c(p_0,p_1,\tilde p) \le c(p_0,p_1,p)+10$.
\end{remark}

\section{Tools: Point configurations of corank $1$}\label{s:corank1}

A point configuration has \defn{corank $1$} if its cardinality exceeds the dimension of its affine span by exactly~$2$ ({\em e.g.} $4$ nonaligned points in $\R^2$, $5$ noncoplanar points in $\R^3$, \ldots). These correspond to the point sets supporting a minimal Radon partition (although we will merely use their affine properties here), and the terminology is motivated by the fact that their dual oriented matroid is of rank~$1$. In this section we spell out a few properties of configurations of corank~$1$ used in the proof of Theorem~\ref{t:mainog}.

\bigskip

The first property reduces the general case of Theorem~\ref{t:mainog} to the special case of configurations of corank~$1$.

\begin{lemma}\label{l:to-standard}
  Let $d \ge 2$ and let $X$ be a set of size at least $d+2$. If $P, Q \in (\R^d)^X$ are full dimensional point configurations in $\R^d$ that are not directly affinely equivalent, then there exists $Y \subseteq X$ of size $d+2$ such that $P_{|Y}$ has corank~$1$ and is not directly affinely equivalent to $Q_{|Y}$.
\end{lemma}
\begin{proof}
  Since $P$ is full dimensional, it contains an affine basis $p_{i_0},\dots,p_{i_d}$. Let $Y_0 = \{i_0,i_1, \ldots, i_d\}$. For any one-element extension $Y \supset Y_0$, the restriction $P_{|Y}$ has corank~$1$. Now, if $Q_{|Y_0}$ is not a basis with the same orientation as $P_{|Y_0}$, then adding any element of $X\setminus Y_0$ to $Y_0$ produces a suitable set $Y$. Otherwise, let $\phi$ be the unique affine transformation such that $q_{i_j}=\phi(p_{i_j})$ for $0\leq j \leq d$. Observe that $\phi$ must be direct. If $P$ and $Q$ are not directly affinely equivalent, there must be some $k\in X$ such that $q_{k}\neq \phi(p_k)$, and thus $Y \eqdef Y_0 \cup \{k\}$ is a suitable set.
\end{proof}

The analysis of degeneracies will use the following elementary fact from oriented matroid theory:

\begin{lemma}\label{l:degenwitness}
  If $P$ is a point configuration of corank~$1$, then every point $p \in P$ lies in at most one inclusion-minimal flat spanned by points of $P\setminus \{p\}$.
\end{lemma}
\begin{proof}
  Let $p \in P$. If $\aff{P\setminus\{p\}}$ is $(d-1)$-dimensional then $p$ lies on no flat spanned by $P\setminus\{p\}$. Otherwise, $P\setminus\{p\}$ is affinely independent, and the statement follows from the observation that the intersection of any two flats spanned by points of $P\setminus\{p\}$ is a flat spanned by points of $P\setminus\{p\}$.
\end{proof}

Let $P \in (\R^d)^X$ be a full-dimensional point configuration of corank~$1$. We call a pair $I \subset X$ \defn{good} for $P$ if it satisfies three conditions: $\aff{P_{|(X\setminus I)}}$ is a hyperplane, $\aff{P_{|I}}$ is a line, and this line is not parallel to this hyperplane. If $I$ is good for $P$, then the line $\aff{P_{|I}}$ intersects the hyperplane $\aff{P_{|(X\setminus I)}}$ in a single point, which we denote $p_I$. When $I$ is good for $P$, we define  $P_{\downarrow I} \eqdef P_{|(X\setminus I)} \sqcup \{p_I\}$ and note that if $P$ has corank~$1$, then $P_{\downarrow I}$ also has corank~$1$ and has one point less than $P$. See Figure~\ref{fig:good} for an example.

\begin{figure}[htpb]
  \centering
  \includegraphics[width=.7\linewidth]{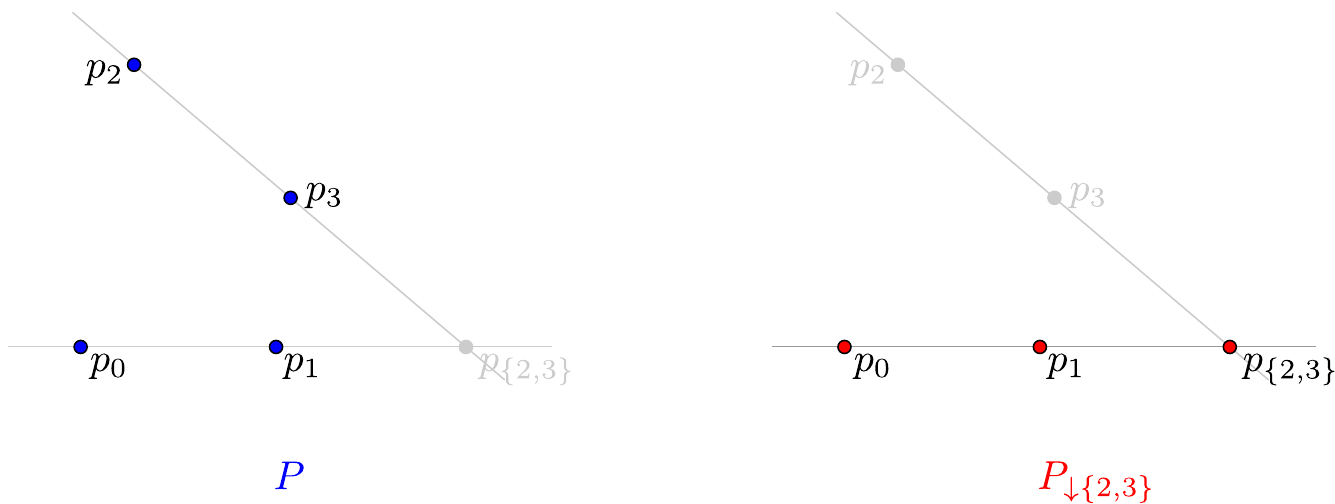}
  \caption{The pair $\{2,3\}$ is good for this configuration, which allows for the construction of $P_{\downarrow \{2,3\}}$.}\label{fig:good}
\end{figure}

\begin{lemma}\label{l:good}
  Let $d\ge 2$, let $P\in (\R^d)^X$ be a full-dimensional point configuration with corank~$1$, and $I \subset X$ a good pair for $P$. If $Q\in (\R^d)^X$ is a point configuration such that $P$ and $Q$ have the same extensions, then $Q$ has corank~$1$, $I \subset X$ is good for $Q$, and $P_{\downarrow I}$ and $Q_{\downarrow I}$ have the same extensions.
\end{lemma}

\noindent
In the conclusion that ``$P_{\downarrow I}$ and $Q_{\downarrow I}$ have the same extensions'', we refer to extensions within their $(d-1)$-dimensional affine hulls, and to $(d-1)$-dimensional chirotopes. In order for this to make sense, we have to prescribe an orientation for the hyperplanes $\aff{P_{\downarrow I}}$ and $\aff{Q_{\downarrow I}}$, so that we can properly assign orientations to the $(d-1)$-simplices that define the chirotopes. We do it in a compatible way by imposing that the orientation of the $(d-1)$-dimensional simplices with vertices $\{p_i\}_{i\in X\setminus I}$ and $\{q_i\}_{i\in X\setminus I}$ coincide.

\begin{proof}
  The fact that $Q$ has the same extensions as $P$ has several consequences. First, $P$ and $Q$ have the same chirotope, so $Q$ also has corank~$1$ and $\aff{Q_{|(X\setminus I)}}$ is a hyperplane that does not contain the line $\aff{Q_{|I}}$. Second, the chirotope of $P\sqcup \{p_I\}$ is realizable on top of $Q$, so the hyperplane $\aff{Q_{|(X\setminus I)}}$ does intersect the line $\aff{Q_{|I}}$. That intersection is therefore a single point $q_I$, and the pair $I$ is good for $Q$.

\medskip

Now consider an extension $\widehat P \in (\R^d)^Y$ of $P_{\downarrow I}$. Let us extend $\widehat P$ into $\widehat P' \eqdef \widehat P \cup P_{|I}$ to turn it into an extension of $P$. We then take a point configuration $\widehat Q'$ that realizes the chirotope of $\widehat P'$ on top of $Q$, and finally restrict $\widehat Q'$ into $\widehat Q \eqdef Q'_{|Y} $. By construction, $\widehat Q$ has the same chirotope as $\widehat P$ and contains $Q_{|(X \setminus I)}$. We must also have $q_I$ in $\widehat Q'$, and therefore in $\widehat Q$, as $\widehat P'$ contains $P$ and $p_I = \aff{P_{|I}} \cap \aff{P_{|(X\setminus I)}}$, and $q_I$ is the only point in $\aff{Q_{|I}} \cap \aff{Q_{|(X\setminus I)}}$. Hence, $\widehat Q$ realizes the chirotope of $\widehat P$ on top of $Q_{\downarrow I}$. Since $Q$ has corank~$1$ and $I$ is good for $Q$, we can exchange the roles of $P$ and $Q$ in this argument to obtain that $P_{\downarrow I}$ and $Q_{\downarrow I}$ have the same extensions.
\end{proof}

\section{Intermezzo: Asymptotic rigidity from two-sided, possibly nongeneric, extensions}

At this point, we already have all the ingredients to prove the following weaker version of Theorem~\ref{t:mainog}:

\begin{proposition}\label{p:main}
   For $d \ge 2$, two full-dimensional finite point sequences in
   $\R^d$ have the same extensions if and only if they are directly
   affinely equivalent
\end{proposition}

\noindent
Note that this statement allows nongeneric extensions and makes symmetric assumptions on the two point configurations. We believe that this simpler proof is of independent interest so we present it here, but this section can be skipped towards the proofs of the stronger Theorems~\ref{t:mainog} and~\ref{t:maineps}. We start with a lemma that enables a recurrence on dimension.

\begin{lemma}\label{l:recur}
  Let $d \ge 2$ and let $P, Q \in (\R^d)^X$ be full-dimensional point configurations of corank~$1$ with the same chirotope. If $P$ and $Q$ are not directly affinely equivalent, then there exists a pair $I \subset X$ that (i) is good for exactly one of them, or (ii) is good for both of them and such that $P_{\downarrow I}$ and  $Q_{\downarrow I}$ are not affinely equivalent.
\end{lemma}
\begin{proof}
    Up to applying the same relabeling to both configurations, we can assume that $X=\{0,1,\dots,d,d+1\}$, that $\{p_0,\dots, p_d\}$ form a direct affine basis, and that $p_{d+1}$ lies in the affine hull of $p_0,p_1, \ldots, p_k$ with $k\le d$, and not in any strictly smaller subset of these points. Since $P$ and $Q$  have the same chirotope, $\{q_0,\dots, q_d\}$ form a direct affine basis, and $q_{d+1}$ lies in the affine hull of $q_0,q_1, \ldots, q_k$ with $k\le d$, and not in the affine hull of any strictly smaller subset of these points. By applying a suitable direct affine transformation to $Q$, we can further assume that $q_i=p_i$ for $0\leq i\leq d$. 

\medskip

With these assumptions, $P$ and $Q$ are not directly affinely equivalent if and only if $p_{d+1}\neq q_{d+1}$. In particular, we cannot have $k=0$ as otherwise $p_{d+1}=p_0 = q_0 = q_{d+1}$. If $k \le d-1$, then the point $p_{d+1}$ belongs both to the line $\aff{\{p_d,p_{d+1}\}}$ and to the hyperplane $\aff{\{p_0,p_1, \ldots, p_{d-1}\}}$. That is, taking $I=\{d,d+1\}$, we have $\aff{P_{|I}} \cap \aff{P_{|(X\setminus I)}} =\{p_{d+1}\}$. It follows that $I$ is good for $P$, that $p_I=p_{d+1}$, $P_{\downarrow I}=P\setminus\{p_d\}$, $q_I=q_{d+1}$, and $Q_{\downarrow I}=Q\setminus\{q_d\}$. Since $p_{d+1}\neq q_{d+1}$, this shows that $P_{\downarrow I}$ and $Q_{\downarrow I}$ are not affinely equivalent. 

\medskip

It remains to handle the case $k=d$. For $1 \le i \le d$, the flat $\aff{P\setminus\{p_i,p_{d+1}\}}$ is a hyperplane and $\aff{\{p_i,p_{d+1}\}}$ is a line. We put $\widehat p_i \eqdef \aff{\{p_i,p_{d+1}\}} \cap \aff{P\setminus\{p_i,p_{d+1}\}}$ if they intersect, and $\widehat p_i \eqdef \infty$ if they are parallel (the line cannot be contained in the hyperplane since $P$ is full-dimensional). We claim that $p_{d+1}$ is uniquely determined by $\{p_0,p_1, \ldots, p_d, \widehat p_1, \widehat p_2, \ldots, \widehat p_d\}$. Indeed, we are in one of three cases:

\begin{itemize}
\item[(i)] There are two indices $1\leq i\neq j\leq d$, such that $\widehat p_i\neq \infty\neq \widehat p_j$. Since $k=d$, the points $p_i,p_j,p_{d+1}$ cannot be colinear, so the lines $\aff{\{p_i,\widehat p_i\}}$ and $\aff{\{p_j,\widehat p_j\}}$ intersect in a single point which is $p_{d+1}$.
  
\item[(ii)] There are two indices $1\leq i\neq j\leq d$, such that $\widehat p_i\neq \infty= \widehat p_j$. In that case, $p_{d+1}$ is on the line $\aff{\{p_i,\widehat p_i\}}$ and on the hyperplane $\Pi$ through $p_j$ and parallel to $\aff{P\setminus\{p_j,p_{d+1}\}}$. We cannot have $p_j \in \aff{P\setminus\{p_j,p_{d+1}\}}$ since $P\setminus\{p_{d+1}\}$ is a basis of $\R^d$. Thus, $\Pi \neq \aff{P\setminus\{p_j,p_{d+1}\}}$ and we cannot have $p_i \in \Pi$ either. It follows that $\aff{\{p_i,\widehat p_i\}}$ intersects $\Pi$ in a single point, which is~$p_{d+1}$.
 		
\item[(iii)] $\widehat p_i= \infty$ for all $1\leq i\leq d$. In this case, $p_{d+1}$ belongs to the hyperplane $\Pi_i$ that contains $p_i$ and is parallel to $\aff{P\setminus\{p_i,p_{d+1}\}}$ for all $1\leq i\leq d$. Since $(p_0,p_1, \ldots, p_d)$ is a basis of $\R^d$, the intersection of the hyperplanes $\aff{P\setminus\{p_i,p_{d+1}\}}$ for $i=1,2, \ldots d$ is $\{p_0\}$, so the intersection of $\Pi_1$,$\Pi_2$, \ldots, $\Pi_d$  is a single point; each contains $p_{d+1}$, so that intersection is~$\{p_{d+1}\}$.
	\end{itemize}
We proceed similarly for $Q$. For $1 \le i \le d$ we let $\widehat q_i$ denote the intersection point of the line $\aff{\{q_i,q_{d+1}\}}$  and the hyperplane $\aff{Q\setminus\{q_i,q_{d+1}\}}$ if they are not parallel, and let $\widehat q_i = \infty$ otherwise. The same analysis yields that $q_{d+1}$ is uniquely determined by $\{q_0,q_1, \ldots, q_d, \widehat q_1, \widehat q_2, \ldots, \widehat q_d\}$. Finally, if $p_{d+1}\neq q_{d+1}$ then there is some $i$ such that $\widehat p_i \neq \widehat q_i$. Observe that $I=\{i,d+1\}$ is good for $P$ (resp. for $Q$) if and only if  $\widehat p_i \neq \infty$ (resp. $\widehat q_i\neq\infty$). If exactly one of $\widehat p_i$ or $\widehat q_i$ is $\infty$ we are in case (i), and otherwise $p_{\{i,d+1\}}=\widehat p_{i}\neq \widehat q_{i}=q_{\{i,d+1\}}$ ensures that $P_{\downarrow\{i,d+1\}}$ and $Q_{\downarrow\{i,d+1\}}$ are not affinely equivalent (because $q_0=p_0, \ldots, q_{i-1}=p_{i-1}, q_{i+1}=p_{i+1}, \ldots, q_{d}=p_{d}$ is a common affine basis).
\end{proof}

We can now prove our first rigidity result.

\begin{proof}[Proof of Proposition~\ref{p:main}]
  For the reverse implication, consider two point configurations $P, Q \in (\R^d)^X$ and let $\phi$ be an invertible, direct affine transformation of $\R^d$ that maps $P$ to $Q$. Observe that $\phi$ maps any extension $\widehat P$ of $P$ to an extension $\widehat Q$ of $Q$. Since invertible direct affine transforms preserve orientations, we must have $\chi_{\widehat Q} = \chi_{\widehat P}$.

\medskip

For the direct implication, we will show that if $P$ and $Q$ are not directly affinely equivalent, then they do not have the same extensions. Observe first that we can assume that $\chi_P=\chi_Q$, as otherwise $P$ and $Q$ do not have the same $0$-point extensions. This implies first of all that $P$ and $Q$ are also not affinely equivalent, as if $P$ and $Q$ were affinely equivalent but not directly affinely equivalent, then $\chi_P =-\chi_Q$.  Therefore, it also implies that $|X|>d+1$, as all configurations of $d+1$ points are affinely equivalent. We claim now that we can assume that $P$ has corank~$1$. Indeed, by Lemma~\ref{l:to-standard} there exist $Y \subset X$ of size $|Y|=d+2$ such that $P_{|Y}$ has corank~$1$ and is not directly affinely equivalent to $Q_{|Y}$. By Lemma~\ref{l:extrest}, if $P_{|Y}$ and $Q_{|Y}$ do not have the same extensions, then neither do $P$ and $Q$.

\medskip

Therefore, to prove the announced statement it suffices to establish the following property for every $d \ge 2$:

\begin{quote}
(Prop$_d$) {\em For every full-dimensional point configurations $P, Q \in (\R^d)^X$ with $P$ of corank~$1$, if $P$ and $Q$ are not affinely equivalent then they do not have the same extensions.}
\end{quote}

\noindent
We proceed by induction on $d$.

\medskip

For the \underline{initialization}, consider two point configurations $P, Q \in (\R^2)^X$, with $P$ of corank~$1$, that are not affinely equivalent. 
By Lemma~\ref{l:recur}, after relabeling and exchanging the roles of $P$ and $Q$ if necessary, we can assume that
$\{p_0,p_1,p_2\}$ is an affine basis and the pair $\{2,3\}$ is good for $P$. If $\{q_0,q_1,q_2\}$ is not an affine basis, then $\chi_P\neq \chi_Q$ and we are done. Otherwise, the Lemma guarantees that we are in one of following cases:
\begin{itemize}
\item[(i)] The pair $\{2,3\}$ is not good for $Q$. This means that 
 the line $\aff{p_2p_3}$ intersects the line $\aff{p_0p_1}$ in a point $p$, whereas the lines $\aff{q_2q_3}$ and $\aff{q_0q_1}$ do not intersect. Then the chirotope of the extension
$\widehat P = P \cup \{p\}$, cannot be realized on top of $Q$. Hence, $P$ and $Q$ do not have the same extensions.

\item[(ii)] 
The pair $\{2,3\}$ is good for $Q$, and the configurations $P'\eqdef \{p_0,p_1,p_2,p_{\{2,3\}}\}$ and $Q'\eqdef \{q_0,q_1,q_2,q_{\{2,3\}}\}$ are not affinely equivalent. 
Let $\phi$ be the affine map that sends $q_i$ to $p_i$ for $i=0,1,2$. We necessarily have $\phi(q_{\{2,3\}})\neq p_{\{2,3\}}$. We can apply Lemma~\ref{l:asym2witness} on the three aligned points $p_0, p_1, p_{\{2,3\}}$ to obtain an extension $\widehat P'$ of $P'$ whose chirotope cannot be realized on top of $\phi(Q')$. Since chirotopes are invariant under affine transformations, the chirotope of $\widehat P'$ cannot be realized on top of $Q'$ either. Then the extension $\widehat P\eqdef \widehat P' \cup \{p_3\}$ of $P$ cannot be realized on top of $Q$, with an argument analogous to that of Lemma~\ref{l:good}.
\end{itemize}

\medskip

Now, for the \underline{propagation}, let $d \ge 3$ and suppose that (Prop$_{d-1}$) holds. Let $P, Q \in (\R^d)^X$ with $P$ of corank~$1$, and suppose that $P$ and $Q$ are not directly affinely equivalent. By Lemma~\ref{l:recur}, we are in one of two cases:
\begin{itemize}
\item[(i)]there exists a pair $I \subset X$ such that $I$ is good for exactly one of $P$ or $Q$, say $Q$ (the argument is symmetric if it is $P$). This means that the line spanned by $Q_{|I}$ and the hyperplane spanned by $Q_{|X\setminus I}$ intersect in some point $q_I$, whereas the line spanned by $P_{|I}$ and the hyperplane spanned by $P_{|X\setminus I}$ are parallel. Letting $\widehat Q = Q \cup \{q_I\}$, we see that $\chi_{\widehat Q}$ cannot be realized on top of $P$. Hence, $P$ and $Q$ do not have the same extensions.

\item[(ii)] $I$ is good for $P$ and $Q$ and the point configurations $P_{\downarrow I}$ and $Q_{\downarrow I}$ are not affinely equivalent. By Lemma~\ref{l:good}, each of $P_{\downarrow I}$ and $Q_{\downarrow I}$ has corank~$1$. By the induction hypothesis (Prop$_{d-1}$), $P_{\downarrow I}$ and $Q_{\downarrow I}$ do not have the same extensions, and 
by Lemma~\ref{l:good}, it follows that $P$ and $Q$ do not have the same extensions.\qedhere
\end{itemize}
\end{proof}

\section{Tools: Scatterings}\label{s:scattering}

To prove that finite generic extensions suffice to discriminate between point configurations that are not affinely equivalent, we follow the steps of the proof of Proposition~\ref{p:main} and, at key places, apply so-called {\em scattering} arguments. The key idea behind scatterings is that the chirotope of a point set records some intersections between the affine and convex hulls of its subsets (but not all, see Figure~\ref{fig:nonsymmetric}). We start with a simple example.

\begin{lemma}\label{l:chirointer}
  Let $P, Q \in (\R^d)^X$ be two point configurations with the same chirotope. Let $Y, Z$ be disjoint subsets of $X$. %
  The affine hull of $P_{|Y}$ intersects the convex hull of $P_{|Z}$ if and only if the affine hull of $Q_{|Y}$ intersects the convex hull of $Q_{|Z}$.
\end{lemma}
\begin{proof}
  Let $F$ denote the affine hull of $P_{|Y}$ and $C$ the convex hull of $P_{|Z}$. The statement is trivial if $F=\R^d$. If $F$ is $(d-1)$-dimensional, the statement is also straightforward since the fact that $F$ intersects $C$ is encoded in the chirotope $\chi_P = \chi_Q$. So we can assume that $F$ has dimension $d-2$ or less. Since $F$ and $C$ are convex and $C$ is compact, they are disjoint if and only if there exists a hyperplane $H$ that separates them strictly. If it exists, such a hyperplane $H$ is parallel to $F$, and can therefore be translated into a hyperplane $H'$ that contains $F$ and is disjoint from $C$. This hyperplane $H'$ can be rotated around successive $d-2$ flats containing $F$ to produce a hyperplane $H"$ that supports $C$ and is spanned by $F$ and a face of $C$. The existence of $H"$ depends only on the chirotope of $P$ and the statement follows. 
\end{proof}

Let us spell out a notion of scattering tailored to our needs. Let $P \in (\R^d)^X$ be a point configuration. For $s \in X$, a \defn{scattering} of $p_s$ in $P$ is a point configuration $S \in (\R^d)^Y$, for some $Y \supset X \setminus \{s\}$, such that
\begin{itemize}
\item[(a)] $S$ is a generic extension of $P_{|X\setminus \{s\}}$ %
\item[(b)] for $j \in Y\setminus X$, the point $s_j$ is on the same side as $p_s$ with respect to every hyperplane spanned by $P_{|X\setminus\{s\}}$, 
\item[(c)] every realization $R \in (\R^d)^Y$ of $\chi_S$ admits a one-point extension $\widehat R = R \cup \{r_s\}$ such that $\widehat R_{|X}$ realizes $\chi_P$.
\end{itemize}
We use two sufficient conditions for the existence of scatterings. Here is the first one:

\begin{lemma}\label{l:scattering1}
  Let $P \in (\R^d)^X$ be a full-dimensional point configuration and $s \in X$. If $p_s$ lies on exactly one inclusion-minimal flat spanned by points of $P\setminus\{p_s\}$, then there exists a scattering of $p_s$ in $P$.
\end{lemma}
\begin{proof}
  Without loss of generality let us assume that $X \cap [d+1] = \emptyset$. Let $\h_0$ denote the union of the (finitely many) hyperplanes spanned by $P_{|X\setminus \{s\}}$ and not containing $p_s$. We fix $\varepsilon>0$ such that the ball $A\eqdef \ball{\varepsilon}{p_s}$ of radius $\varepsilon$ centered in $p_s$ is disjoint from $\h_0$. For $i =1, 2, \ldots, d$ we pick $s_i \in A \setminus \h_{i-1}$ and define $\h_i$ to be the union of the (finitely many) hyperplanes spanned by points of $P\sqcup\{s_j\}_{j\le i}$. Finally, we choose $s_{d+1} \in A\setminus \h_d$ so that the simplex spanned by $\{s_j\}_{j\in [d+1]}$ contains $p_s$ in its interior. We let $Y=(X \setminus \{s\})\cup [d+1]$ and define  $S=P_{|X\setminus\{s\}} \cup \{s_j\}_{j \in [d+1]}$.
  
\medskip

Let us check that $S$ is a scattering of $p_s$ in $P$. Condition~(a) holds because each $s_i$ avoids $\h_{i-1}$, and Condition~(b) holds because each $s_i$ is in $A$. It remains to check Condition~(c). For this, consider some arbitrary realization $R$ of $\chi_S$. Let $\mathcal{Z}$ be the set of subsets $Z \subset X\setminus\{s\}$ such that $\aff{P_{|Z}}$ is the unique inclusion-minimal flat spanned by $P\setminus\{p_s\}$ that contains $p_s$. Let us fix some $Z_0 \in \mathcal{Z}$. Since $\aff{P_{|Z_0}}$ contains $p_s$, it intersects $\conv(S_{|[d+1]})$, and, by Lemma~\ref{l:chirointer}, the flat $\aff{R_{|Z_0}}$ therefore intersects $\Delta = \conv(R_{|[d+1]})$. We can therefore extend $R$ into $\widehat R = R \cup \{r_s\}$ by picking $r_s$ to be any point in $(\aff{R_{|Z_0}} \cap \Delta)\setminus \h$, where $\h$ denotes the union of all flats $F$ spanned by points of $R$ such that $F \not\supset \aff{R_{|Z_0}}$.

We claim that $\widehat R_{|X}$ realizes $\chi_P$. Indeed, consider some $T \in X^{d+1}$. We are in one of three cases:
\begin{itemize}
\item If $T$ avoids $s$, then $\chi_{\widehat R}(T) = \chi_P(T)$ because $R$ realizes $S$ and $P_{|X\setminus \{s\}} = S_{|X\setminus \{s\}}$.
\item If $T$ contains $s$ and $\chi_P(T)\neq 0$, then $\conv(S_{|[d+1]})$ avoids $\aff{P_{|T}}$ (because $S_{|[d+1]} \subset A$), and, by Lemma~\ref{l:chirointer}, $\Delta= \conv(R_{|[d+1]})$ therefore avoids $\aff{R_{|T}}$. It follows that $\chi_{\widehat R}(T) = \chi_{\widehat R}(T')$, where $T'$ is the (d+1)-tuple obtained by substituting $1$ for $s$ in $T$. Since $\chi_R=\chi_S$, $\chi_{\widehat R}(T') = \chi_S(T')$, and since $s_i \in A$ we finally have $\chi_S(T') = \chi_P(T)$.
\item If $T$ contains $s$ and $\chi_P(T) = 0$, then there must exist some $Z \in \mathcal{Z}$ such that $T$ contains every element of $Z$, and $r_s \in \aff{R_{Z}}$ ensures that $\chi_{\widehat R}(T)=0$ as well.
\end{itemize}
Altogether, condition~(c) is verified and $S$ is a scattering of $p_s$ in $P$.
\end{proof}

Here is the second sufficient condition for the existence of a scattering that we need:

\begin{lemma}\label{l:scattering2}
Let $P \in (\R^d)^X$ be a full-dimensional point configuration and $s \in X$. If $p_s$ lies on exactly two inclusion-minimal flats spanned by points of $P_{|X\setminus\{s\}}$ and these flats intersect in exactly $\{p_s\}$, then there exists a scattering of $p_s$ in $P$.
\end{lemma}
\begin{proof}
Let $Z_1$ and $Z_2$ be the sets of labels of the points of $P_{|X\setminus\{s\}}$ spanning the two inclusion-minimal flats containing $p_s$, which we denote $F_1 $ and $F_2$, respectively. We construct our scattering in two stages.

\bigskip

First, for $i \in \{1,2\}$, we proceed like in the proof of Lemma~\ref{l:scattering1} to pick a simplex $\sigma_i \in (\R^d)^{S_i}$ in $F_i$ that contains $p_s$ in its relative interior. We choose $\sigma_i$ small enough to avoid every hyperplane spanned by $P$ that does not contain $p_s$, and otherwise generic in $F_i$. Let $P' = P_{|Z_1 \cup Z_2} \cup \sigma_1 \cup \sigma_2$.

\medskip

We claim that in every realization $R$ of $\chi_{P'}$, the simplices spanned by $R_{|S_1}$ and $R_{|S_2}$ intersect in a single point, which we denote by $r_s$. Indeed, by two applications of Lemma~\ref{l:chirointer}, we know that $\aff{R_{|Z_1}}$ intersects the interior of $\conv(R_{|S_2})$, and conversely $\aff{R_{|Z_2}}$ intersects the interior of $\conv(R_{|S_1})$. Moreover, any affine dependency in a point configuration translates into the vanishing of some orientations, so the chirotope of a full-dimensional point configuration records the dimension of the affine hull of each of its subsets. It follows that $\aff{R_{|Z_1}}$ and $\aff{R_{|Z_2}}$ are either disjoint or intersect in a single point. It follows that $\aff{R_{|Z_1}} \cap \conv(R_{|S_2})$ coincides with $\aff{R_{|Z_2}} \cap \conv(R_{|S_1})$, and this point is the intersection of the two simplices.

\medskip

At this stage, $P_{|X\setminus\{s\}} \cup \sigma_1\cup \sigma_2$ fulfills properties (b) and (c) expected from a scattering of $p_s$ in $P$. Since every point in $\sigma_i$ lies in $F_i$, it does not, however, satisfy property (a).

\bigskip

Let us fix some $i\in \{1,2\}$. By the genericity of the choice of $\sigma_i$ in $F_i$, its vertices belong to a single inclusion-minimal flat spanned by the configuration, namely $F_i$. We can therefore apply Lemma~\ref{l:scattering1} to scatter, one by one, each of these points in $P_{|X\setminus\{s\}} \cup \sigma_1\cup \sigma_2$. In this way, we obtain a new configuration $\widetilde P$ that has the desired properties. Indeed, we get (a) from \ref{l:scattering1}, and (b) because we just replaced the points of $\sigma_i$ (which were in the neighborhood of $p_s$), by points in their neighborhood. Finally, to get (c), observe that Lemma~\ref{l:scattering1} guarantees that after the last scattering operation we can always recover a realization $R$ of $P_{|X\setminus \{s\}} \cup \sigma_1\cup \sigma_2$. And, as already discussed, this realization can be extended by the intersection $r_s \eqdef \aff{R_{|Z_1}}\cap \aff{R_{|Z_2}}$ which is on the required side of each hyperplane spanned by $R_{|X\setminus \{s\}}$ as it belongs to the convex hull of $R_{|S_1 \cup S_2}$.
\end{proof}

\section{Recursive construction for very generic configurations}\label{s:verygeneric}

For $d \ge 1$, we say that a configuration $P=\{p_0,\dots,p_{d},p_{d+1}\}$ in $\R^d$ is \defn{very generic} if the three following conditions hold:
\begin{itemize}
\item[(i)] $P$ is generic, 
\item[(ii)] $P$ has corank~$1$,
\item[(iii)] $d=1$ or for every $1\leq i\leq d$ the pair $I=\{p_i,p_{d+1}\}$ is good and $P_{\downarrow I}$ is very generic. 
\end{itemize}
If $P$ is very generic, then for every $1\leq i\leq d$ we define $\widehat p_i \eqdef \aff{\{p_i,p_{d+1}\}} \cap \aff{P\setminus\{p_i,p_{d+1}\}}$. Observe that in a very generic configuration $P=\{p_0,\dots,p_{d},p_{d+1}\}$ of dimension~$d\geq 2$, the point $p_{d+1}$ can be reconstructed as the intersection of the lines $\aff{\{p_i,\widehat p_i\}}$ and  $\aff{\{p_j,\widehat p_j\}}$ for any indices $1 \le i < j \le d$.

\begin{figure}[htpb]
	\centering
	\includegraphics[width=.35\linewidth]{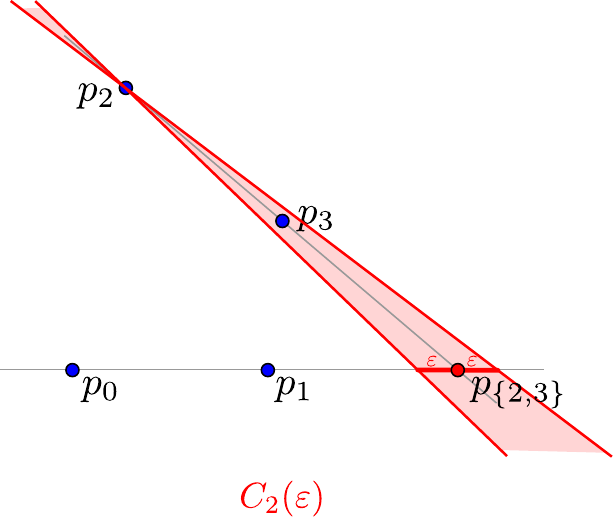}	
	\caption{Example of the definition of $C_2(P,\varepsilon)$.}\label{fig:doublecone}
\end{figure}

\bigskip

Consider a very generic configuration $P=\{p_0,\dots,p_{d},p_{d+1}\}$. It is straightforward that for any generic point~$q$ sufficiently close to $p_{d+1}$, the configuration $\{p_0,\dots,p_{d},q\}$ is also very generic and the point $\widehat q_i \eqdef  \aff{\{p_i,q\}} \cap \aff{P\setminus\{p_i,p_{d+1}\}}$ is close to $\widehat p_i$. To formulate a reverse implication, for $\varepsilon>0$ we define $C_i(P,\varepsilon)$ to be the union of lines through $p_i$ that intersect $\ball{\varepsilon}{\widehat p_i}\cap \aff{P\setminus\{p_i,p_{d+1}\}}$, where $\ball{r}{c}$ denotes the ball with center $c$ and radius~$r$. In other words, $C_i(P,\varepsilon)$ is the set of points $q$ such that $\widehat q_i \eqdef  \aff{\{p_i,q\}} \cap \aff{P\setminus\{p_i,p_{d+1}\}}$ lies within distance $\varepsilon$ from~$\widehat p_i$.

\begin{lemma}\label{l:doublecone}
  For any very generic configuration $P$ in $\R^d$, $d \ge 2$, there exist two constants $\tau(P)>0$ and $\kappa(P)$ such that for every $\varepsilon < \tau(P)$, for every $1\leq i< j\leq d$ we have $C_i(P,\kappa(P)\varepsilon)\cap C_j(P,\kappa(P)\varepsilon)\subset \ball{\varepsilon}{p_{d+1}}$.
\end{lemma}
\begin{proof}
Let $x,y$ be distinct points in $\R^d$. For any real $\alpha \in [0,\pi]$ we let $C(x,y,\alpha)$ denote the union of lines through $x$ that make an angle at most $\alpha$ with $\aff{\{x,y\}}$. For any real $r \ge 0$ we let $D(x,y,r)$ denote the set of points within distance $r$ of the line $\aff{\{x,y\}}$. We note that  $C(x,y,\alpha)$ is a cone of revolution with apex $x$ and axis $\aff{\{x,y\}}$, and that $D(x,y,r)$ is a cylinder of revolution with axis $\aff{\{x,y\}}$.

\begin{figure}[htpb]
	\centering
	\includegraphics[width=.85\linewidth]{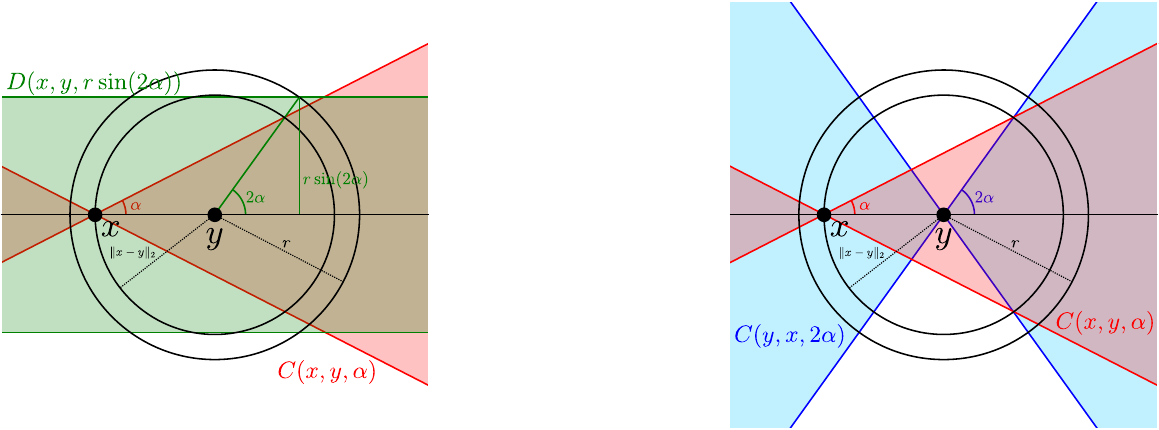}	
	\caption{Sketch for Equation~\ref{eq:ccb}}\label{fig:coneproof}
\end{figure}

By the inscribed angle theorem and elementary planar geometry arguments (see Figure~\ref{fig:coneproof}), we have that for every $r \ge \|x-y\|_2$,
\begin{equation}\label{eq:ccb}
\pth{C(x,y,\alpha) \cap B_{r}(y)} \subset D\pth{x,y,r\sin(2\alpha)} 
\qquad \text{and} \qquad 
\pth{C(x,y,\alpha) \setminus B_{r}(y)} \subset C(y,x,2\alpha).\end{equation}
We now fix distinct indices $1 \le i, j \le d$ and use the above inclusions to analyze the intersection of $C_i(P,\delta)$ and $C_j(P,\delta)$ for $\delta$ small enough. Specifically, we let $r \eqdef \max(\|p_i-p_{d+1}\|_2,\|p_j-p_{d+1}\|_2)$ and let $\alpha \eqdef \sin^{-1}(\delta/r)$. Since the points $p_i,p_j$ and $p_{d+1}$ are pairwise distinct, $\alpha$ is well-defined and $\alpha \to 0$ as $\delta \to 0$. Note that for $k \in \{i,j\}$ we have $C_k(P,\delta) \subset C(p_k,p_{{d+1}},\alpha)$, so that
\[ \begin{array}{rcl}
C_i(P,\delta) \cap B_{r}(p_{d+1}) & \subset & D\pth{{p_i,p_{d+1}},r\sin(2\alpha)}\\
C_j(P,\delta) \cap B_{r}(p_{d+1}) & \subset & D\pth{{p_j,p_{d+1}},r\sin(2\alpha)}\\
C_i(P,\delta) \setminus B_{r}(p_{d+1}) & \subset & C(p_{d+1},p_i,2\alpha)\\
C_j(P,\delta) \setminus B_{r}(p_{d+1}) & \subset & C(p_{d+1},p_j,2\alpha)\\
\end{array}
\]
Since $P$ is very generic, the points $p_i$, $p_j$ and $p_{d+1}$ are not aligned. It follows that for $\alpha$ small enough, the cones $C(p_{d+1},p_i,2\alpha)$ and $C(p_{d+1},p_j,2\alpha)$ intersect in exactly $\{p_{d+1}\}$. This implies that for $\delta$ small enough, the regions $C_i(P,\delta)$ and $C_j(P,\delta)$ do not intersect outside $B_r(p_{d+1})$ and their intersection is contained in the intersection $D({p_i,p_{d+1}},r\sin(2\alpha)) \cap D({p_j,p_{d+1}},r\sin(2\alpha))$. For any fixed $\varepsilon>0$, when $\alpha>0$ is small enough we have 
\begin{equation}\label{eq:DcapDinB}
    D({p_i,p_{d+1}},r\sin(2\alpha)) \cap D({p_j,p_{d+1}},r\sin(2\alpha)) \subset B_\varepsilon(p_{d+1}),
\end{equation} 
A careful analysis of the dependencies between $\alpha$, $\delta$, and $\varepsilon$ shows that whenever $\varepsilon$ is smaller than a constant that depends only on $P$, then there is a constant $\kappa$ depending only on $P$, such that for $\delta\leq \kappa \varepsilon$ we have that $\alpha = \sin^{-1}(\delta/r)$ satisfies Equation~\eqref{eq:DcapDinB}. The statement follows.
\end{proof}

\begin{lemma}\label{l:eps}
For any $d \ge 2$ and any very generic configuration $R = \{p_0,\dots,p_{d},p_{d+1}\}$ with $p_0,\dots,p_{d}$ an affine basis, there exist constants $\tau(R)>0$ and $c(R)$ such that the following holds. For any point configuration $P\in(\R^d)^X$ containing $R$ and for every $0<\varepsilon<\tau(R)$, there exists a generic extension of $P$ of size at most $c(R)\log\frac1{\varepsilon}$ whose chirotope cannot be realized on top of $\{p_0,p_1, \dots,p_{d},q_{d+1}\}$ for any $q_{d+1}$ with $\|p_{d+1}-q_{d+1}\|_2>\varepsilon$.
\end{lemma}
\begin{proof}
By Lemma~\ref{l:doublecone}, there exist $\tau(R)$ and $\kappa(R)$ such that for any $\varepsilon < \tau(R)$ and any 
$0<\delta<\kappa(R)\varepsilon$ we have $C_1(P,\delta)\cap C_2(P,\delta)\subseteq \ball{\varepsilon}{p_{d+1}}$. We first prove the statement without the complexity analysis, using an induction on $d$.

\bigskip

\underline{Base case} ($d=2$). The points $p_0,p_2,\widehat p_1$ are distinct and aligned, so Lemma~\ref{l:asym2witness} yields a constructible extension $\widehat P_{-2}$ of $\{p_0,p_1,p_2,\widehat p_1\}$ whose chirotope cannot be realized on top\footnote{Note that for the sake of readability, here and below we treat $\widehat \cdot_i$ as a label.} of $\{p_0,p_2,\widehat q_1\}$ for any $\widehat q_1\notin \ball{\delta}{\widehat p_1}$. Similarly, $p_0$, $p_1$ and $\widehat p_2$ are distinct and aligned, and Lemma~\ref{l:asym2witness} yields a constructible extension $\widehat P_{-1}$ of $\widehat P_{-2} \cup \{\widehat p_2\}$ whose chirotope cannot be realized on top of $\{p_0,p_1,\widehat q_2\}$ for any $\widehat q_2\notin \ball{\delta}{\widehat p_2}$. We let $\widehat P_0 = \widehat P_{-1} \cup \{p_3\}$. 

We next scatter $\widehat P_0$ by repeated applications of Lemmas~\ref{l:scattering1} and~\ref{l:scattering2}. Let us call a point $p_i \in \widehat P_0$ \defn{undesirable} if it is not in $P$ and lies on at least one line spanned by two points of $\widehat P_0$ with lower indexes. We let $k$ denote the number of undesirable points of $\widehat P_0$, and denote their labels $i_1 > i_2 > \ldots > i_k$. For $1 \le j \le k$, we define $\widehat P_j$ to be a scattering of $p_{i_j}$ in $\widehat P_{j-1}$, as given by Lemmas~\ref{l:scattering1} or~\ref{l:scattering2} (according to the number of lines that contain $p_{i_j}$ and are spanned by points of $\widehat P_j$ with indexes lower than $i_j$). We end up with some point set $\widehat P_\varepsilon \eqdef \widehat P_k$ which, we claim, is a generic extension of $P$. Indeed, on the one hand, property~(a) of scatterings ensures that any point that is not generic in $\widehat P_\varepsilon$ must come from $\widehat P_0$, and on the other hand, any three points in $\widehat P_0$ that are aligned and not all in $P$ contain at least one undesirable point, which is therefore no longer in $\widehat P_\varepsilon$.

We claim that the chirotope of $\widehat P_\varepsilon$ cannot be realized on top of $\{p_0,p_1,p_2,q_3\}$ for any $q_3 \notin B_\varepsilon(p_3)$. Indeed, when $\{p_0,p_1,p_2,q_3\}$ is not very generic, this follows from the fact that $\widehat P$ contains $\widehat p_1$ and $\widehat p_2$; when $\{p_0,p_1,p_2,q_3\}$ is very generic, this follows from the fact that $\widehat q_1 \notin \ball{\delta}{\widehat p_1}$ or $\widehat q_2 \notin \ball{\delta}{\widehat p_2}$ by Lemma~\ref{l:doublecone}.

\bigskip

\underline{Inductive step} ($d \ge 3$). Since $R$ is very generic, $R_1 \eqdef \{p_0,\widehat p_1,p_2,p_3, \ldots, p_d\}$ also is, and in particular it has corank~$1$. By induction, within $\aff{R_1}$ there is a generic extension $\widehat R_1$ of the $(d-1)$-dimensional configuration $R_1$ that cannot be realized on top of $(R_1\setminus \{\widehat p_1\})\cup \{\widehat q_1\}$ whenever $\|\widehat p_1-\widehat q_1\|_2>\varepsilon$, that is whenever $q_{d+1} \notin C_1(\delta)$. Since $\widehat R_1$ is a generic extension of $R_1$ (within $\aff{R_1}$), we can assume up to perturbation within $\aff{R_1}$, that $\aff{R_1}$ is the only inclusion-minimal flat spanned by $P$ that contains any point of $\widehat R_1 \setminus R_1$. 

We next scatter the points of $\widehat P_{-1} \eqdef P \cup \widehat R_1$. Thanks to the perturbation applied above, $\widehat P_{-1}$ only fails to be a generic extension of $P$ in two ways. First, each point $p$ of $\widehat P_{-1}\setminus R_1$ lies on the hyperplane $\aff{R_1}$; this is the only flat spanned by $\widehat P_{-1}\setminus \{p\}$ that contains $p$, so we can apply Lemma~\ref{l:scattering1} to scatter $p$ in $\widehat P_{-1}$; successively treating each point $p$ of $\widehat P_{-1}\setminus R_1$ in this way yields a generic extension $\widehat P_0$ of $P\sqcup \{\widehat p_1\}$ (in $\R^d$). Next, the point $\widehat p_1$ lies on the hyperplane $\aff{R_1}$ and on the line $\aff{\{p_1,p_{d+1}\}}$; since $R_1$ has corank~$1$, Lemma~\ref{l:degenwitness} ensures that $\widehat p_1$ is on a unique inclusion-minimal flat spanned by points of $R_1\setminus\{\widehat p_1\}$, so Lemma~\ref{l:scattering2} yields a scattering $\widehat P_1$ of $\widehat p_1$ in $\widehat P_0$. Now, $\widehat P_1$ is a generic extension of $P$ and it retains the property that its chirotope cannot be realized on top of $\{p_0,p_1, \dots,p_{d},q_{d+1}\}$  for any $q_{d+1}$ with $q_{d+1} \notin C_1(\delta)$.

Similarly, we can construct a generic extension $\widehat P_\varepsilon$ of $\widehat P_1$ whose  chirotope cannot be realized on top of $\{p_0,p_1, \dots,p_{d},q_{d+1}\}$  for any $q_{d+1}$ with $q_{d+1} \notin C_2(\delta)$. By Lemma~\ref{l:doublecone}, for every $q_{d+1} \notin B_\varepsilon(p_{d+1})$, there exists $i \in \{1,2\}$ such that $q_{d+1} \notin C_i(\delta)$. The chirotope of $\widehat P_\varepsilon$ can therefore not be realized on top of $\{p_0,p_1, \dots,p_{d},q_{d+1}\}$ for any $q_{d+1}$ for any $q_{d+1} \notin B_\varepsilon(p_{d+1})$. This concludes the induction step.

\bigskip

It remains to bound the size of $\widehat P_\varepsilon\setminus P$. In the base case, we apply twice Lemma~\ref{l:asym2witness}, each contributing at most $O(\log\frac1{\delta})=O(\log\frac1{\varepsilon})$ points. Similarly, in the inductive step we apply twice the result in dimension $d-1$, each contributing at most $O(\log\frac1{\delta})=O(\log\frac1{\varepsilon})$ points.  This adds up to the announced bound.
\end{proof}

\begin{remark}\label{r:perturbhighd}
Following Remark~\ref{r:perturb}, note that the bound of $20\log\frac1{\varepsilon}+O(1)$ for the number of points added in each application of Lemma~\ref{l:asym2witness} is also valid if we slightly perturb $p_{d+1}$ within its neighborhood. Therefore, there is a threshold $\tau = \tau(R_{p_{d+1}})$ and a constant $\tilde c(R_{p_{d+1}})$ such that for every $\varepsilon \le \tau$ and for every $g \in \ball{\varepsilon}{p_{d+1}}$, we have $c(R_g)\leq \tilde c(R_{p_{d+1}})$, where $R_g= \{p_0,\dots,p_{d},g\}$.
\end{remark}

\section{Perturbations of configurations that are not very generic}\label{s:perturb}

We now turn our attention to point configurations that are {\em not} very generic. Let $S=(s_1,s_2,\dots,s_{d-1})$ be a sequence of points in $\R^d$ and $p \in \R^d$. If $\aff{\{s_1,s_2,\dots,s_{d-1},p\}}$ is a hyperplane, we put
\[ H_{S,p}^+=\{q\in\R^d\ : \ \chi(s_1,s_2,\dots,s_{d-1},p,q)>0\} \quad \text{and}\quad H_{S,p}^-=\{q\in\R^d\ : \ \chi(s_1,s_2,\dots,s_{d-1},p,q) < 0\}\]
for the two {\em open} halfspaces it bounds, and we define for $\varepsilon>0$ 
\begin{align*}
	H_{S,p}^+(\varepsilon)&\eqdef 	\bigcup\nolimits_{p'\in H_{S,p}^-\cap B_{\varepsilon}(p)} H_{S,p'}^+,&
	H_{S,p}^-(\varepsilon)&\eqdef 	\bigcup\nolimits_{p'\in H_{S,p}^+\cap B_{\varepsilon}(p)} H_{S,p'}^-.
\end{align*}
Note that if $\aff{\{s_1,s_2,\dots,s_{d-1},p\}}$ is a hyperplane, then the projection $\pi:\R^d \to \R^2$ parallel to $R \eqdef \aff{\{s_1,s_2,\dots,s_{d-1}\}}$ maps $H_{S,p}^\pm(\varepsilon)$ to the union of two open half-planes bounded by lines through $\pi(R)$, see Figure~\ref{fig:wedgedhalfspaces}.

\begin{figure}[htpb]
\centering
\includegraphics[width=.85\linewidth]{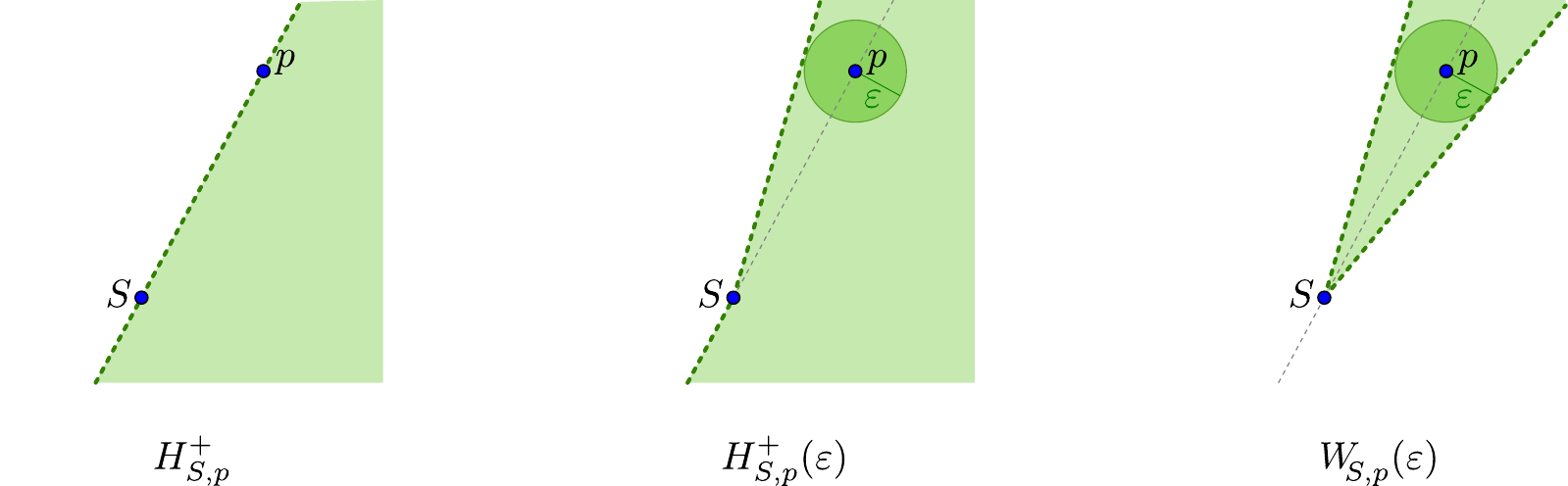}
\caption{The images of $H_{S,p}^+$, $H_{S,p}^+(\varepsilon)$, and $W_{S,p}(\varepsilon)$ through the projection parallel to $\aff{\{s_1,s_2,\dots,s_{d-1}\}}$.\label{fig:wedgedhalfspaces}}
\end{figure}

Now, given a {\em set} $S$ of $d-1$ points in $\R^d$ and a point $p \notin \aff{S}$, we fix some order on $S$ and define the {\em wedge} $W_{S,p}(\varepsilon) \eqdef H_{S,p}^-(\varepsilon) \cap H_{S,p}^+(\varepsilon)$ for any $\varepsilon>0$. Note that $W_{S,p}(\varepsilon)$ does not depend on the choice of order on $S$.

\begin{lemma}\label{l:hyperplanewedges}
  For any full-dimensional configuration $P$ in $\R^d$ with $d \ge 2$, and $p\in P$,  there exist two constants $\tau(P)>0$ and $\kappa(P)$ such that for every $0<\varepsilon < \tau(P)$,
  we have  $\bigcap_{S\in \mathcal{S}}W_{S,p}(\kappa(P)\varepsilon)\subset \ball{\varepsilon}{p}$, where $\mathcal{S}$ is the set of all affinely independent $(d-1)$-element subsets of $P\setminus \{p\}$.
\end{lemma}
\begin{proof}
Since $P$ is full dimensional, $\{p\}$ can be completed to an affine basis $\{p\}\cup B\subset P$ of $\R^d$. We let $\mathcal{B}$ denote the set of $(d-1)$-element subsets of $B$, and note that $\mathcal{B}\subseteq \mathcal{S}$. Note that $p$ is the intersection of the hyperplanes spanned by $p$ and the elements of $\mathcal{B}$, that is $\bigcap_{S \in \mathcal{B}} \aff{S \cup \{p\}} = \{p\}$. This ensures that there exists $\alpha(P)>0$ such that such that every point $q\in \sphere{1}{p}$ is at distance at least $\alpha(P)$ from some hyperplane $\aff{S \cup \{p\}}$ with $S \in \mathcal{B}$. For any $\varepsilon$, applying a scaling of factor $\varepsilon$ and center $p_{d+1}$, we get that for every point $q\in \sphere{\varepsilon}{p}$, there is some $S \in \mathcal{B}$ such that the distance from $q$ to $\aff{S \cup \{p\}}$ is at least $\alpha(P)\varepsilon$. In other words, the intersection of the $(\alpha(P)\varepsilon)$-neighborhoods of $H_{S,p}$ for $S \in \mathcal{B}$ contains $\{p\}$ but is disjoint from $\sphere{\varepsilon}{p}$; that intersection is therefore  contained in $\ball{\varepsilon}{p}$.

\begin{figure}[htpb]
\centering
\includegraphics[width=.35\linewidth]{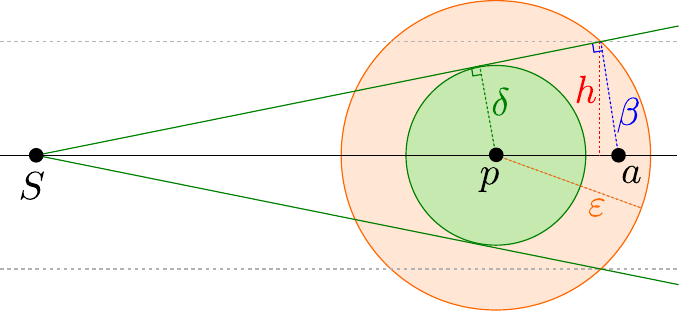}
\caption{This depicts the projection parallel to $\aff{S}$.
Since $\dist(p,\aff{S})>\varepsilon$, we have $\|S-p\|_2 >\varepsilon$ and the point $a$ lies inside $\ball{\varepsilon}{p}$, that is $\dist(p,a)<\varepsilon$. We then have $h= \delta\frac{\|S-a\|_2}{\|S-p\|_2} \le \delta\frac{\|S-p\|_2+\varepsilon}{\|S-p\|_2} \le 2\delta$.\label{fig:thales}}
\end{figure}

Now, let $\tau(P) \eqdef \min_{S \in {\mathcal B}} \dist(p,\aff{S})$ where $\dist(p,F)$ is the distance from a point $p$ to a flat $F$. Fix some $0 < \varepsilon \le \tau(P)$ and let $\delta \eqdef \frac{\alpha(P)}2\varepsilon$. Considering the projection parallel to $\aff{S}$ for some $S \in \mathcal{B}$ (Figure~\ref{fig:thales}) reveals that $\ball{\varepsilon}{p}\cap W_{S,p}(\delta)$ is contained in the $(2\delta)$-neighborhood of $H_{S,p}$. It follows that $\bigcap_{S\in \mathcal{B}}W_{S,p}(\alpha(P)\varepsilon)$ is contained in the intersection of the $(\alpha(P)\varepsilon)$-neighborhoods of $H_{S,p}$ for $S \in \mathcal{B}$, which does not intersect $\sphere{\varepsilon}{p}$. 
\end{proof}

\begin{lemma}\label{l:hyperplanewedgesmonotonicity}
Let $S=(s_1,s_2,\dots,s_{d-1})$ be a sequence of points in $\R^d$ and $p, q\in\R^d$ be such that $\{s_1,s_2,\dots,s_{d-1},p, q\}$ is affinely independent. If $q\notin H_{S,p}^+(\varepsilon)$, then $H_{S,p}^+(\varepsilon) \cap H_{S,p}^- \,\subset\, H_{S,q}^+$. If $q\notin H_{S,p}^-(\varepsilon)$, then $H_{S,p}^-(\varepsilon) \cap H_{S,p}^+ \,\subset\, H_{S,q}^-$.
\end{lemma}

\begin{figure}[htpb]
\centering
\includegraphics[width=.25\linewidth]{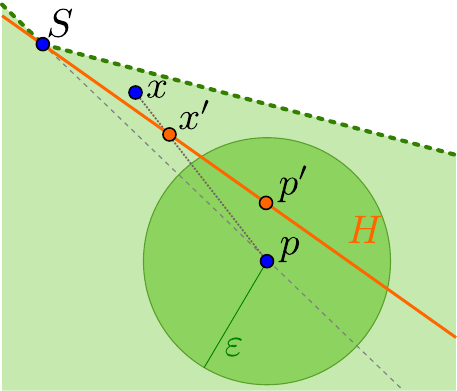}
\caption{Sketch of the proof of Lemma~\ref{l:hyperplanewedgesmonotonicity}.\label{fig:wedgeproof}}
\end{figure}

\begin{proof}
We prove the first statement, so suppose that $q\notin H_{S,p}^+(\varepsilon)$. Note that this implies that $p \in H_{S,q}^+$. Now consider some point $x\in H_{S,p}^+(\varepsilon) \cap H_{S,p}^-$. If $x\notin H_{S,q}^+$, then the hyperplane $H \eqdef \aff{\{s_1,s_2,\dots,s_{d-1},q\}}$ separates $p$ from $x$ (here we allow $x \in H$). Let $x'$ be the point of $H$ on the segment joining $x$ to $p$. Since the region $H_{S,p}^+(\varepsilon) \cap H_{S,p}^-$ is convex, contains $x$, and has $p$ on its boundary, it must also contain the point $x'$. There is therefore a point $p' \in B_\varepsilon(p)$ such that $x'\in \aff{\{s_1,s_2,\dots,s_{d-1},p'\}}$. This forces $p' \in H$, which in turns implies that $q \in \overline{H_{S,p'}^+} \subset H_{S,p}^+(\varepsilon)$, a contradiction. See Figure~\ref{fig:wedgeproof} for a sketch. The proof of the second statement is similar, {\em mutatis mutandis}.
\end{proof}

We now have all the ingredients to generalize Lemma~\ref{l:eps} to arbitrary point configurations.

\begin{lemma}\label{l:epsarb}
For any $d \ge 2$ and any configuration $R = \{p_0,\dots,p_{d},p_{d+1}\}$ with $p_0,\dots,p_{d}$ an affine basis, there exist constants $\tilde c(R)$ and $\tau(R)>0$ such that the following holds. For any point configuration $P\in(\R^d)^X$ containing $R$ and for every $0 < \varepsilon < \tau(R)$, there exists a generic extension of $P$ of size at most $\tilde c(R)\log\frac1{\varepsilon}$ whose chirotope cannot be realized on top of $\{p_0,p_1, \dots,p_{d},q_{d+1}\}$ for any $q_{d+1}$ with $\|p_{d+1}-q_{d+1}\|_2>\varepsilon$.
\end{lemma}
\begin{proof}
We focus on the case where the subconfiguration is not very generic, as otherwise the statement is Lemma~\ref{l:eps}. There is some $i$ such that $\{p_0,\dots,p_{d},p_{d+1}\}\setminus\{p_i\}$ is full dimensional. We let $\mathcal{S}$ be the set of all $(d-1)$-element subsets of $\{p_0,\dots,p_{d}\}\setminus\{p_i\}$ (which are all affinely independent). Let $\tau(R)>0$ and $\kappa(R)$ be the constants from Lemma~\ref{l:hyperplanewedges} and fix $0< \varepsilon < \tau(R)$. We let $\delta \eqdef \kappa(R)\varepsilon$ and we have $\bigcap_{S\in \mathcal{S}}W_{S,p_{d+1}}(\delta)\subset \ball{\varepsilon}{p_{d+1}}$.

\bigskip

We let $\gamma \eqdef \frac{\delta}3$, we fix some arbitrary order on $\mathcal{S} \eqdef \{S_1,S_2, \ldots, S_d\}$, and we note that as $p_0,p_1, \dots,p_{d}$ is an affine basis, the set of points $x$ such that $\{p_0,p_1, \dots,p_{d},x\}$ is a very generic configuration is dense in $\R^d$. For each $1 \le i \le d$ we can choose some point $g_i^+ \in \ball{\varepsilon}{p_{d+1}}$ such that  $\{p_0,p_1, \dots,p_{d},g_i^+\}$ is very generic and $\ball{\gamma}{g_i^+} \subset H_{S_i,p_{d+1}}^+(\delta) \cap H_{S_i,p_{d+1}}^-$. (Refer to Figure~\ref{fig:inclusiongh}.) Similarly, we pick $g_i^- \in \ball{\varepsilon}{p_{d+1}}$ such that  $\{p_0,p_1, \dots,p_{d},g_i^-\}$ is very generic and $\ball{\gamma}{g_i^-} \subset H_{S_i,p_{d+1}}^-(\delta) \cap H_{S_i,p_{d+1}}^+$.

\begin{figure}[htpb]
\centering
\includegraphics[width=.25\linewidth]{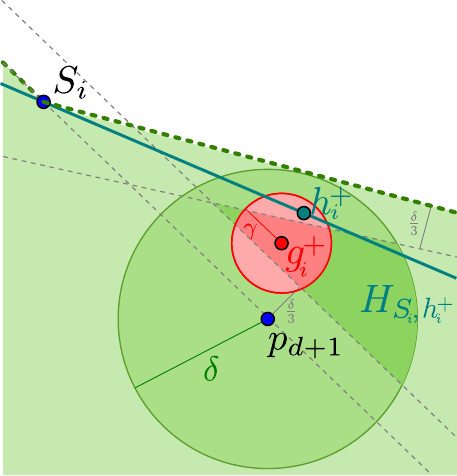}
\caption{The inclusion $H^+_{S_i,h_i^+} \subseteq H^+_{S_i,p_{d+1}}(\delta)$.\label{fig:inclusiongh}}
\end{figure}

We now let $\widehat P_0 \eqdef P$ and for $i=1,2, \ldots, d$ we build some generic extension $\widehat P_i$ of $\widehat P_{i-1}$ as follows. We first apply Lemma~\ref{l:eps} to obtain a generic extension $\widehat P_{i-\frac12}$ of $\widehat P_{i-1} \cup \{g_i^-,g_i^+\}$ whose chirotope cannot be realized on top of $\{p_0,p_1, \dots,p_{d},h_i^-\}$ for any $h_i^-$ with $\|g_i^--h_i^-\|_2>\gamma$. We then again apply Lemma~\ref{l:eps} to obtain a generic extension $\widehat P_{i}$ of $\widehat P_{i-\frac12}$ whose chirotope cannot be realized on top of $\{p_0,p_1, \dots,p_{d},h_i^+\}$ for any $h_i^+$ with $\|g_i^+-h_i^+\|_2>\gamma$. We let $\widehat P_\varepsilon \eqdef \widehat P_d$.

\medskip

Now, suppose that the chirotope of $\widehat P_\varepsilon$ has a realization $R$ on top of $\{p_0,p_1, \dots,p_{d},q_{d+1}\}$ for some point $q_{d+1}$. The construction forces  $R$ to contain in particular some points $h_1^-, h_1^+, h_2^-, h_2^+, \ldots, h_k^-, h_k^+$ such that for every $i \in [d]$, we have $\|g_i^+-h_i^+\|_2 \le \gamma$ and $\|g_i^--h_i^-\|_2 \le \gamma$. Moreover, since $g_i^+ \in H^-_{S_i,p_{d+1}}$ we must have $p_{d+1} \in H^+_{S_i,g_i^+}$, and since $R$ realizes the chirotope of $\widehat P_\varepsilon$ we must also have $q_{d+1} \in H^+_{S_i,h_i^+}$. (Here we used that all points with indices in $S_i$ are the same in $R$ and $P$.) Altogether, $q_{d+1}$ must be in $H^+_{S_i,h_i^+} \subseteq H^+_{S_i,p_{d+1}}(\delta)$ as well as in $H^-_{S_i,h_i^-} \subseteq H^-_{S_i,p_{d+1}}(\delta)$. In other words, $q_{d+1}$ must belong to $\bigcap_{i=1}^k W_{S_1,p_{d+1}}(\delta)$ which is contained in $\ball{\varepsilon}{p_{d+1}}$ by Lemma~\ref{l:hyperplanewedges}.

\bigskip

The point configuration $\widehat P_\varepsilon$ is constructed through $2d$ applications of Lemma~\ref{l:eps}. 
Each application takes a configuration $R_g \eqdef \{p_0,p_1,\ldots,p_d,g\}$ for some point $g \in \ball{\varepsilon}{p_{d+1}}$ to play the role of the subconfiguration~$R$ in the notations of Lemma~\ref{l:eps}.
Therefore, it adds at most $c(R_g)\log\frac1{\gamma}$ extension points, where $c(\cdot)$ is the constant given by Lemma~\ref{l:eps}. By Remark~\ref{r:perturbhighd} there is a threshold $\tau = \tau(R_{p_{d+1}})$ and a constant $\tilde c(R_{p_{d+1}})$
such that for every $0<\varepsilon \le \tau$ and for every $g \in \ball{\varepsilon}{p_{d+1}}$, we have $c(R_g)\leq \tilde c(R_{p_{d+1}})$. Altogether, for $\varepsilon \le \tau$, $\widehat P_\varepsilon \setminus P$ has size at most $2d\tilde c(R_{p_{d+1}})\log\frac1{\gamma} \le \tilde c(R)\log\frac1{\varepsilon}$ for some constant $\tilde c(R)$, because $\gamma=\frac{1}{3}\delta=\frac{ \kappa(P)}{3}\varepsilon$. 
\end{proof}

\section{Proof of Theorems~\ref{t:mainog} and~\ref{t:maineps}}\label{s:proofs}

We can finally complete the proof of our main theorems. We start with our rigidity property:

\setcounter{reptheorem}{0}
\begin{reptheorem}
Two full-dimensional point configurations $P$ and $Q$ in $\R^d$ are  directly affinely equivalent if and only if for every finite generic extension $\widehat P$ of $P$, the chirotope of $\widehat P$ is realizable on top of $Q$.
\end{reptheorem}
\begin{proof}
  The direct implication follows from the fact that any invertible direct affine transformation of $\R^d$ preserves orientations, and therefore if it maps $P$ to $Q$, then it transports every extension of $P$ to an extension of $Q$ with the same chirotope. For the reverse implication, consider two full-dimensional configurations $P,Q \in (\R^d)^X$ that are not directly affinely equivalent. Up to relabeling, we fix some basis $\{p_0,p_1, \ldots, p_d\} \subseteq P$ in $\R^d$. If $P$ and $Q$ do not have the same chirotope, then the chirotope of $P$ is already not realizable on top of $Q$. We therefore assume that they do have the same chirotope. This has two consequences: (i) since they are not directly affinely equivalent, they are not affinely equivalent, and (ii) $\{q_0,q_1, \ldots, q_d\}$ is also a basis of $\R^d$. We consider the affine map $\varphi$ that sends each $q_i$ to $p_i$, $0 \le i \le d$. Since $\varphi(Q) \neq P$, there exists some $i \in X$ such that $\varepsilon \eqdef \|p_i-\varphi(q_i)\|_2 >0$. By Lemma~\ref{l:epsarb}, there exists an extension $\widehat P_\varepsilon$ of $P$ whose chirotope cannot be realized on top of $\varphi(Q)$. This chirotope cannot be realized on top of $Q$ either.
\end{proof}

We can now prove our uniform quantitative statement:

\setcounter{reptheorem}{1}
\begin{reptheorem}
  For every $d \ge 2$ and every full-dimensional point configuration $P \in (\R^d)^X$ there exists constants $C(P)$ and $\tau(P)>0$ such that for every $0<\varepsilon \le \tau(P)$, there exists a finite generic extension $\widehat P$ of $P$, with $\widehat P \setminus P$ of size at most $C(P)\log\frac1{\varepsilon}$ with the following property: for any configuration $Q \in (\R^d)^X$ such that $\chi_{\widehat P}$ can be realized on top of $Q$, there exists a direct affine transform $\varphi$ such that $\max_{i \in X} \|p_i-\varphi(q_i)\|_2 \le \varepsilon$.
\end{reptheorem}
\begin{proof}
Relabeling if needed, we can assume that $X=\{0,1,\dots, n\}$ and that $\{p_0,p_1, \ldots, p_d\}$ is an affine basis. Let $\widehat P_d\eqdef P$, and for every $d+1\leq i \leq n$, apply Lemma~\ref{l:epsarb} to find a generic extension $\widehat P_i$ of $\widehat P_{i-1}$ such that for any $q_i$ with $\|p_{i}-q_{i}\|_2>\varepsilon$, the chirotope of $\widehat P_i$ cannot be realized on top of $\{p_0,p_1, \dots,p_{d},q_{i}\}$. The upper bound on the size of $\widehat P \setminus P$ follows directly from Lemma~\ref{l:epsarb} with $C(P)$ being the sum of the $\tilde c(R)$ from Lemma~\ref{l:epsarb} for $R=\{p_0,p_1, \ldots, p_d, p\}$ with $p$ ranging over $P\setminus\{p_0,p_1, \ldots, p_d\}$, whenever $\varepsilon$ is bounded from above by the minimum of the $\tau(R)$ from Lemma~\ref{l:epsarb} for $R=\{p_0,p_1, \ldots, p_d, p\}$ with $p$ ranging over $P\setminus\{p_0,p_1, \ldots, p_d\}$.

\bigskip

 We claim that $\widehat P\eqdef \widehat P_{n}$ fulfills the desired conditions.  Indeed, let $Q \in (\R^d)^X$ be a configuration for which there is no direct affine transform $\varphi$ such that $\max_{i \in X} \|p_i-\varphi(q_i)\|_2 \le \varepsilon$. If $\{q_0,q_1, \ldots, q_d\}$ is not an affine basis with the same orientation as $\{p_0,p_1, \ldots, p_d\}$, then $P$ and $Q$ do not have the same chirotope and $\chi_{\widehat P}$ is not realizable on top of $Q$. If it is an affine basis with the same orientation, let $\psi$ be the direct affine transform mapping $q_i\mapsto p_i$ for $0\leq i\leq d$. By hypothesis, we have $\|p_i-\psi(q_i)\|_2 > \varepsilon$ for some $i\in X$. In particular, the chirotope of $\widehat P_i$ cannot be realized on top of $\psi(Q)$, and hence it cannot be realized on top of $Q$ either. The same goes for the chirotope of its extension $\widehat P$.
\end{proof}

 Note that given two bases $\{p_0,p_1, \ldots, p_d\}$ and $\{q_0,q_1, \ldots, q_d\}$ of $\R^d$ and some $\varepsilon>0$, the set of affine transforms $\varphi$ of $\R^d$ such that $\|p_i-\varphi(q_i)\|_2 \le \varepsilon$ is compact. With this, Theorem~\ref{t:mainog} also immediately follows from Theorem~\ref{t:maineps}.

\bibliographystyle{plain}
\bibliography{bibliography.bib}

\end{document}